\documentclass[a4paper]{amsart}

\usepackage{graphics,amssymb,enumerate}
\usepackage{hyperref}
\usepackage{graphicx}
\usepackage{color}
\usepackage{float}
\floatstyle{boxed}
\restylefloat{figure}

\newcommand {\bd} {\begin{displaymath}}
\newcommand {\ed} {\end{displaymath}}
\newcommand {\be} {\begin{equation}}
\newcommand {\ee} {\end{equation}}
\newcommand {\bea} {\begin{eqnarray}}
\newcommand {\eea} {\end{eqnarray}}

\newtheorem{lemma}{Lemma}
\newtheorem{theorem}{Theorem}
\newtheorem{proposition}{Proposition}

\newtheorem{example}{Example}
\newtheorem{remark}{Remark}

\def\g{\mathfrak{g}}

\begin{document}

\title[]{On the  characteristic polynomial of  Cartan matrices and Chebyshev polynomials}
\date{}

\author[]{Pantelis A.~Damianou}
\email{damianou@ucy.ac.cy}
\address{Department of Mathematics and Statistics\\
University of Cyprus\\
P.O.~Box 20537, 1678 Nicosia\\Cyprus}

\keywords{Cartan matrix, Chebyshev polynomials, simple Lie algebras, Coxeter polynomial, Cyclotomic polynomials}
\subjclass{ 33C45, 17B20, 20F55}

\thanks{}

\begin{abstract}
We explore some interesting features of the  characteristic polynomial of  the  Cartan matrix of a  complex simple Lie algebra. The characteristic polynomial  is closely related with the Chebyshev polynomials of first and second kind. In addition, we give explicit formulas for the characteristic polynomial of the Coxeter adjacency matrix, we compute the associated polynomials  and use them to derive the Coxeter polynomial of the underlying graph.  We  determine the expression of the Coxeter  and associated polynomials as a product of cyclotomic  factors. We use this data to propose an algorithm for factoring Chebyshev polynomials over the integers. Finally, we prove an interesting sine formula which involves  the exponents, the Coxeter number and the determinant of the Cartan matrix.
\end{abstract}

%%% ----------------------------------------------------------------------
\maketitle
%%% ----------------------------------------

\section{Introduction}

The aim  of this paper is to explore  the intimate connection between Chebyshev polynomials and root systems of complex  simple Lie algebras. Chebyshev polynomials are used to generate  the characteristic and associated polynomials of Cartan and adjacency matrices and conversely one can use machinery  from Lie theory to derive properties of Chebyshev polynomials. Many  of the results in this paper are well-known but we re-derive them in an elementary fashion using  properties of  Chebyshev polynomials. In essence, this paper is mostly a survey but it also includes some new results. For example, the explicit factorization of the Chebyshev polynomials in terms of  the polynomials $\psi_j(x)$ is new.

Cartan matrices appear in the classification of  simple   Lie algebras over the complex numbers.  A Cartan matrix is associated to each such  Lie algebra. It is an $\ell \times \ell$ square matrix where $\ell$ is the rank of the Lie algebra. The Cartan matrix encodes all the properties of the simple Lie algebra it represents.
Let $\mathfrak{g}$ be a simple complex  Lie algebra, $\mathfrak{h}$ a Cartan subalgebra and    $\Pi=\{ \alpha_1, \dots \alpha_{\ell} \}$ a basis of simple roots for the root system $\Delta$ of $\mathfrak{h}$ in
$\mathfrak{g}$.    The elements of the Cartan matrix $C$ are given by

 \be \label{cij}  c_{ij} :=      2 \frac{  \left( \alpha_i, \alpha_j\right) } { \left( \alpha_j, \alpha_j \right)}  \ee

  where the inner product is induced by the Killing form.  The $\ell \times \ell$-matrix $C $  is invertible.  It is called the \emph{Cartan matrix} of $\g$.  The Cartan matrix for a  complex simple Lie algebra obeys  the following properties:
\begin{enumerate}
\item
$C$ is symmetrizable. There exists a diagonal matrix $D$ such that $DC$ is symmetric.
\item
$c_{ii}=2 $.
\item
$ c_{ij} \in \{ 0, -1, -2,-3 \}$ for $i \not=j$.
\item
$c_{ij}=0  \Leftrightarrow c_{ji}=0$.

\end{enumerate}

 The detailed  machinery for constructing the Cartan matrix from the root system can be found e.g. in \cite[p. 55]{humphreys} or \cite[p. 111]{knapp}.  In the following example, we give full details for the case of $sl(4, \mathbb{C})$ which is of type $A_3$.

  \begin{example}

{\rm Let $E$ be the hyperplane  of $\mathbb{R}^4$ for which the coordinates sum to $0$ (i.e. vectors orthogonal to $(1,1,1,1)$). Let $\Delta$ be the set of vectors in $E$ of length $\sqrt{2}$ with integer coordinates.  There are $12$ such vectors in all.  We use the standard inner product in $\mathbb{R}^4$ and the standard orthonormal basis $\{ \epsilon_1, \epsilon_2, \epsilon_3,  \epsilon_4 \}$. Then, it is easy to see that $\Delta = \{ \epsilon_i-\epsilon_j \ | \ i \not= j \}$.  The vectors
\bd
 \begin{array}{lcl}
\alpha_1 & =& \epsilon_1 -\epsilon_2 \\
\alpha_2 & =& \epsilon_2 -\epsilon_3 \\
\alpha_3 & =& \epsilon_3 -\epsilon_4 \
\end{array}
 \ed
form a basis of the root system in the sense that each vector in $\Delta$ is a linear combination of these three vectors with integer coefficients,  either all nonnegative or all nonpositive. For example,  $\epsilon_1 -\epsilon_3=\alpha_1+\alpha_2$, $ \epsilon_2 -\epsilon_4 =\alpha_2+\alpha_3$ and $\epsilon_1-\epsilon_4=\alpha_1+\alpha_2+\alpha_3$.  Therefore $\Pi=\{\alpha_1, \alpha_2, \alpha_3 \}$,  and the set of positive roots $\Delta^{+}$  is given by
\bd \Delta^{+}= \{ \alpha_1, \alpha_2, \alpha_3, \alpha_1+\alpha_2, \alpha_2+\alpha_3, \alpha_1+\alpha_2+\alpha_3  \} \ . \ed
Define the matrix $C$ using (\ref{cij}). It is clear that $c_{ii}=2 $ and

\bd c_{i,i+1}=2 \frac{  (\alpha_i, \alpha_{i+1}) } { (\alpha_{i+1}, \alpha_{i+1}) } =-1  \ \ \ \ \ \ i=1,2  \ . \ed

Similar calculations lead to the following form  of the Cartan matrix
\bd C= \begin{pmatrix}
 2 &  -1 & 0 \cr
  -1&  2 & -1 \cr
  0& -1 & 2 \end{pmatrix} \ .
  \ed

}

\end{example}

The  complex simple Lie algebras are classified  as:
\bd
A_l, B_l, C_l, D_l, E_6, E_7, E_8, F_4, G_2 \ .
\ed
Traditionally, $   A_l, B_l, C_l, D_l   $ are called the classical Lie algebras while  $E_6, E_7, E_8, F_4, G_2  $ are called  the exceptional  Lie algebras.
Moreover, for any Cartan matrix there exists just one simple complex  Lie algebra up to isomorphism
giving rise to it. The classification of simple complex  Lie algebras is due to Killing and Cartan around 1890.  Simple Lie algebras over ${\bf C}$ are classified by using the associated  Dynkin diagram. It is a graph whose vertices correspond to the elements of  $\Pi$. Each pair of vertices
$\alpha_i$, $\alpha_j$ are connected by
\bd
m_{ij}={ 4 (\alpha_i, \alpha_j)^2 \over (\alpha_i,\alpha_i) (\alpha_j, \alpha_j) }
\ed
edges, where
\bd m_{ij} \in \{ 0,1,2,3 \}  \ . \ed

\newpage
\centerline{\bf Dynkin Diagrams for simple Lie algebras}
\bigskip
$
%TCIMACRO{
%\TeXButton{Diagram.An}{\put(25,20){\line(1,0){50}}
%\put(85,20){\line(1,0){50}}
%\put(205,20){\line(1,0){50}}
%\put(20,20){\circle{10}}
%\put(80,20){\circle{10}}
%\put(140,20){\circle{10}}
%\put(200,20){\circle{10}}
%\put(260,20){\circle{10}}
%\put(17,30){$\omega_1$}
%\put(77,30){$\omega_2$}
%\put(137,30){$\omega_3$}
%\put(197,30){$\omega_{n}$}
%\put(257,30){$\omega_{n+1}$}
%\put(45,13){$u_1$}
%\put(105,13){$u_2$}
%\put(225,13){$u_n$}
%\put(160,20){$\ldots$}
%\put(0,50){\QTR{bf}{A}$_{n+1}$}}}%
%BeginExpansion
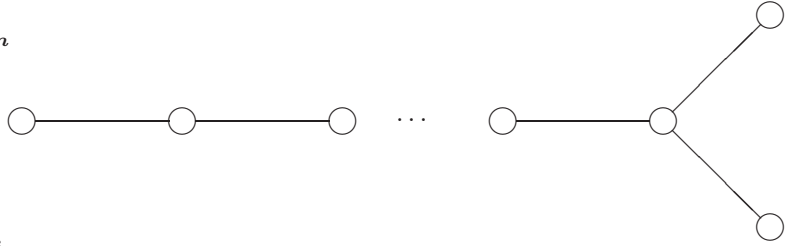
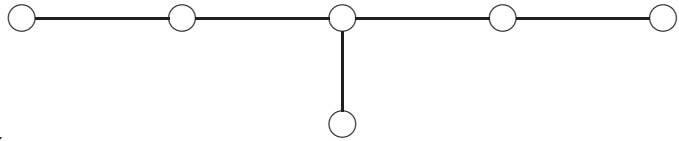
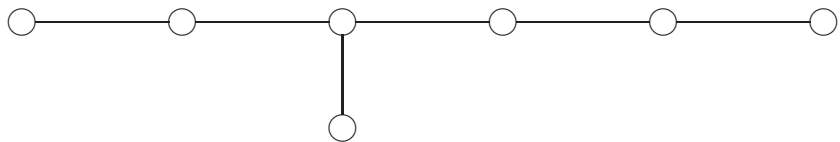
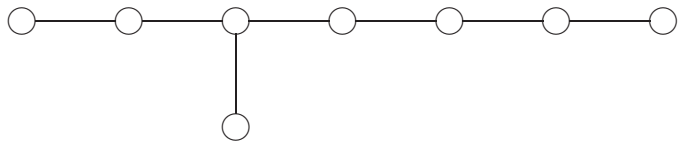
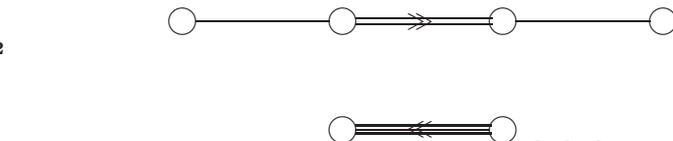
\put(25,20){\line(1,0){50}}
\put(85,20){\line(1,0){50}}
\put(205,20){\line(1,0){50}}
\put(20,20){\circle{10}}
\put(80,20){\circle{10}}
\put(140,20){\circle{10}}
\put(200,20){\circle{10}}
\put(260,20){\circle{10}}
\put(160,20){$\ldots$}
\put(0,50){$\boldsymbol{A_{n}}$}
%EndExpansion
$

%$
%TCIMACRO{
%\TeXButton{line}{\newline
%\put(25,0){\line(1,0){250}}
%\newline}}%
%BeginExpansion
%\newline
%\put(25,0){\line(1,0){250}}
%\newline
%EndExpansion
%$

%\bigskip

$%
%TCIMACRO{
%\TeXButton{Diagram.Bn}{\put(25,20){\line(1,0){50}}
%\put(85,20){\line(1,0){50}}
%\put(205,19){\line(1,0){50}}
%\put(205,21){\line(1,0){50}}
%\put(20,20){\circle{10}}
%\put(80,20){\circle{10}}
%\put(140,20){\circle{10}}
%\put(200,20){\circle{10}}
%\put(260,20){\circle{10}}
%\put(17,30){$\omega_1$}
%\put(77,30){$\omega_2$}
%\put(137,30){$\omega_3$}
%\put(197,30){$\omega_{n}$}
%\put(257,30){$\omega_{n+1}$}
%\put(45,13){$u_1$}
%\put(105,13){$u_2$}
%\put(218,13){$u_n$}
%\put(160,20){$\ldots$}
%\put(0,50){\QTR{bf}{B}$_{n+1}$}
%\put(227,17.5){$\gg $}}}%
%BeginExpansion
\put(25,20){\line(1,0){50}}
\put(85,20){\line(1,0){50}}
\put(205,19){\line(1,0){50}}
\put(205,21){\line(1,0){50}}
\put(20,20){\circle{10}}
\put(80,20){\circle{10}}
\put(140,20){\circle{10}}
\put(200,20){\circle{10}}
\put(260,20){\circle{10}}
\put(160,20){$\ldots$}
\put(0,50){$\boldsymbol{B_{n}}$}
\put(227,17.5){$\gg $}
%EndExpansion
$

$%
%TCIMACRO{
%\TeXButton{Diagram.Cn}{\put(25,20){\line(1,0){50}}
%\put(85,20){\line(1,0){50}}
%\put(205,19){\line(1,0){50}}
%\put(205,21){\line(1,0){50}}
%\put(20,20){\circle{10}}
%\put(80,20){\circle{10}}
%\put(140,20){\circle{10}}
%\put(200,20){\circle{10}}
%\put(260,20){\circle{10}}
%\put(17,30){$\omega_1$}
%\put(77,30){$\omega_2$}
%\put(137,30){$\omega_3$}
%\put(197,30){$\omega_{n}$}
%\put(257,30){$\omega_{n+1}$}
%\put(45,13){$u_1$}
%\put(105,13){$u_2$}
%\put(218,13){$u_n$}
%\put(160,20){$\ldots$}
%\put(0,50){\QTR{bf}{C}$_{n+1}$}
%\put(227,17.5){$\ll $}}}%
%BeginExpansion
\put(25,20){\line(1,0){50}}
\put(85,20){\line(1,0){50}}
\put(205,19){\line(1,0){50}}
\put(205,21){\line(1,0){50}}
\put(20,20){\circle{10}}
\put(80,20){\circle{10}}
\put(140,20){\circle{10}}
\put(200,20){\circle{10}}
\put(260,20){\circle{10}}
\put(160,20){$\ldots$}
\put(0,50){$\boldsymbol{C_{n}}$}
\put(227,17.5){$\ll $}
%EndExpansion
$

$
%TCIMACRO{
%\TeXButton{Diagram.Dn}{\put(25,20){\line(1,0){50}}
%\put(85,20){\line(1,0){50}}
%\put(205,20){\line(1,0){50}}
%\put(263.5,23.5){\line(1,1){33}}
%\put(263.5,16.5){\line(1,-1){33}}
%\put(20,20){\circle{10}}
%\put(80,20){\circle{10}}
%\put(140,20){\circle{10}}
%\put(200,20){\circle{10}}
%\put(260,20){\circle{10}}
%\put(300,60){\circle{10}}
%\put(300,-20){\circle{10}}
%\put(17,30){$\omega_1$}
%\put(77,30){$\omega_2$}
%\put(137,30){$\omega_3$}
%\put(197,30){$\omega_{n-2}$}
%\put(270,20){$\omega_{n-1}$}
%\put(45,13){$u_1$}
%\put(105,13){$u_2$}
%\put(222,13){$u_{n-2}$}
%\put(160,20){$\ldots$}
%\put(0,50){\QTR{bf}{D}$_{n+1}$}
%\put(297,47){$\omega_n$}
%\put(297,-8){$\omega_{n+1}$}
%\put(254,40){$u_{n-1}$}
%\put(262,0){$u_{n}$}}}%
%BeginExpansion
\put(25,20){\line(1,0){50}}
\put(85,20){\line(1,0){50}}
\put(205,20){\line(1,0){50}}
\put(263.5,23.5){\line(1,1){33}}
\put(263.5,16.5){\line(1,-1){33}}
\put(20,20){\circle{10}}
\put(80,20){\circle{10}}
\put(140,20){\circle{10}}
\put(200,20){\circle{10}}
\put(260,20){\circle{10}}
\put(300,60){\circle{10}}
\put(300,-20){\circle{10}}
\put(160,20){$\ldots$}
\put(0,50){$\boldsymbol{D_{n}}$}
%
%EndExpansion
$

$
%TCIMACRO{
%\TeXButton{diagram E6}{\put(25,20){\line(1,0){50}}
%\put(85,20){\line(1,0){50}}
%\put(205,20){\line(1,0){50}}
%\put(145,20){\line(1,0){50}}
%\put(140,15){\line(0,-1){30}}
%\put(20,20){\circle{10}}
%\put(80,20){\circle{10}}
%\put(140,20){\circle{10}}
%\put(200,20){\circle{10}}
%\put(260,20){\circle{10}}
%\put(140,-20){\circle{10}}
%\put(17,30){$\omega_1$}
%\put(77,30){$\omega_3$}
%\put(137,30){$\omega_4$}
%\put(197,30){$\omega_{5}$}
%\put(257,30){$\omega_{6}$}
%\put(150,-20){$\omega_{2}$}
%\put(45,13){$u_1$}
%\put(105,13){$u_3$}
%\put(165,13){$u_4$}
%\put(225,13){$u_5$}
%\put(128,0){$u_2$}
%\put(0,50){\QTR{bf}{E}$_{6}$}
%\put(65,50){$\omega_0 =-(\omega_1+2\omega_2+2\omega_3+3\omega_4+2\omega_5+\omega_6)$}}}%
%BeginExpansion
\put(25,20){\line(1,0){50}}
\put(85,20){\line(1,0){50}}
\put(205,20){\line(1,0){50}}
\put(145,20){\line(1,0){50}}
\put(140,15){\line(0,-1){30}}
\put(20,20){\circle{10}}
\put(80,20){\circle{10}}
\put(140,20){\circle{10}}
\put(200,20){\circle{10}}
\put(260,20){\circle{10}}
\put(140,-20){\circle{10}}
\put(0,50){$\boldsymbol{E_{6}}$}
$

$
%TCIMACRO{
%\TeXButton{diagram E7}{\put(25,20){\line(1,0){50}}
%\put(85,20){\line(1,0){50}}
%\put(205,20){\line(1,0){50}}
%\put(145,20){\line(1,0){50}}
%\put(140,15){\line(0,-1){30}}
%\put(265,20){\line(1,0){50}}
%\put(20,20){\circle{10}}
%\put(80,20){\circle{10}}
%\put(140,20){\circle{10}}
%\put(200,20){\circle{10}}
%\put(260,20){\circle{10}}
%\put(140,-20){\circle{10}}
%\put(320,20){\circle{10}}
%\put(17,30){$\omega_1$}
%\put(77,30){$\omega_3$}
%\put(137,30){$\omega_4$}
%\put(197,30){$\omega_{5}$}
%\put(257,30){$\omega_{6}$}
%\put(150,-20){$\omega_{2}$}
%\put(317,30){$\omega_{7}$}
%\put(45,13){$u_1$}
%\put(105,13){$u_3$}
%\put(165,13){$u_4$}
%\put(225,13){$u_5$}
%\put(128,0){$u_2$}
%\put(285,13){$u_6$}
%\put(0,50){\QTR{bf}{E}$_{7}$}
%\put(65,50){$\omega_0 =-(2\omega_1+2\omega_2+3\omega_3+4\omega_4+3\omega_5+2\omega_6+\omega_7)$}}}%
%BeginExpansion
\put(25,20){\line(1,0){50}}
\put(85,20){\line(1,0){50}}
\put(205,20){\line(1,0){50}}
\put(145,20){\line(1,0){50}}
\put(140,15){\line(0,-1){30}}
\put(265,20){\line(1,0){50}}
\put(20,20){\circle{10}}
\put(80,20){\circle{10}}
\put(140,20){\circle{10}}
\put(200,20){\circle{10}}
\put(260,20){\circle{10}}
\put(140,-20){\circle{10}}
\put(320,20){\circle{10}}
\put(0,50){$\boldsymbol{E_{7}}$}
%
%EndExpansion
$

$
%TCIMACRO{
%\TeXButton{diagram E8}{\put(25,20){\line(1,0){30}}
%\put(65,20){\line(1,0){30}}
%\put(105,20){\line(1,0){30}}
%\put(145,20){\line(1,0){30}}
%\put(185,20){\line(1,0){30}}
%\put(225,20){\line(1,0){30}}
%\put(100,15){\line(0,-1){30}}
%\put(20,20){\circle{10}}
%\put(60,20){\circle{10}}
%\put(100,20){\circle{10}}
%\put(140,20){\circle{10}}
%\put(180,20){\circle{10}}
%\put(220,20){\circle{10}}
%\put(260,20){\circle{10}}
%\put(100,-20){\circle{10}}
%\put(17,30){$\omega_1$}
%\put(57,30){$\omega_3$}
%\put(97,30){$\omega_4$}
%\put(137,30){$\omega_{5}$}
%\put(177,30){$\omega_{6}$}
%\put(78,-20){$\omega_{2}$}
%\put(217,30){$\omega_{7}$}
%\put(257,30){$\omega_{8}$}
%\put(36,13){$u_1$}
%\put(76,13){$u_3$}
%\put(116,13){$u_4$}
%\put(156,13){$u_5$}
%\put(88,0){$u_2$}
%\put(196,13){$u_6$}
%\put(236,13){$u_7$}
%\put(0,50){\QTR{bf}{E}$_{8}$}
%\put(65,50){$\omega_0 =-(2\omega_1+3\omega_2+4\omega_3+6\omega_4+5\omega_5+4\omega_6+3\omega_7+2\omega_8)$}
%\put(150,-15){$\underline{\QTR{bf}{BV}\ \QTR{bf}{E}_{\QTR{bf}{8}}}$}}}%
%BeginExpansion
\put(25,20){\line(1,0){30}}
\put(65,20){\line(1,0){30}}
\put(105,20){\line(1,0){30}}
\put(145,20){\line(1,0){30}}
\put(185,20){\line(1,0){30}}
\put(225,20){\line(1,0){30}}
\put(100,15){\line(0,-1){30}}
\put(20,20){\circle{10}}
\put(60,20){\circle{10}}
\put(100,20){\circle{10}}
\put(140,20){\circle{10}}
\put(180,20){\circle{10}}
\put(220,20){\circle{10}}
\put(260,20){\circle{10}}
\put(100,-20){\circle{10}}
\put(0,50){$\boldsymbol{E_{8}}$}
%EndExpansion
$

$
%TCIMACRO{
%\TeXButton{diagram F4}{\put(85,0){\line(1,0){50}}
%\put(205,0){\line(1,0){50}}
%\put(145,1){\line(1,0){51}}
%\put(145,-1){\line(1,0){51}}
%\put(80,0){\circle{10}}
%\put(140,0){\circle{10}}
%\put(200,0){\circle{10}}
%\put(260,0){\circle{10}}
%\put(77,10){$\omega_1$}
%\put(137,10){$\omega_2$}
%\put(197,10){$\omega_3$}
%\put(257,10){$\omega_{4}$}
%\put(100,-7){$u_1$}
%\put(173,-7){$u_2$}
%\put(220,-7){$u_3$}
%\put(0,30){\QTR{bf}{F}$_{4}$}
%\put(164,-2.5){$\gg $}
%\put(65,30){$\omega_0 =-(2\omega_1+3\omega_2+4\omega_3+2\omega_4)$}}}%
%BeginExpansion
\put(85,0){\line(1,0){50}}
\put(205,0){\line(1,0){50}}
\put(145,1){\line(1,0){51}}
\put(145,-1){\line(1,0){51}}
\put(80,0){\circle{10}}
\put(140,0){\circle{10}}
\put(200,0){\circle{10}}
\put(260,0){\circle{10}}
\put(0,30){$\boldsymbol{F_{4}}$}
\put(164,-2.5){$\gg $}
%EndExpansion
$

%$%
%TCIMACRO{
%\TeXButton{line}{\put(25,10){\line(1,0){250}}
%\newline}}%
%BeginExpansion
%\put(25,10){\line(1,0){250}}
%\newline%
%EndExpansion
%$

$
%TCIMACRO{
%\TeXButton{diagram G2}{\put(145,0){\line(1,0){50}}
%\put(145,1.5){\line(1,0){50.5}}
%\put(145,-1.5){\line(1,0){50.5}}
%\put(140,0){\circle{10}}
%\put(200,0){\circle{10}}
%\put(137,10){$\omega_1$}
%\put(197,10){$\omega_2$}
%\put(173,-7){$u_2$}
%\put(0,30){\QTR{bf}{G}$_{2}$}
%\put(164,-2.5){$\ll $}
%\put(65,30){$\omega_0 =-(3\omega_1+2\omega_2)$}}}%
%BeginExpansion
\put(145,0){\line(1,0){50}}
\put(145,1.5){\line(1,0){50.5}}
\put(145,-1.5){\line(1,0){50.5}}
\put(140,0){\circle{10}}
\put(200,0){\circle{10}}
\put(0,30){$\boldsymbol{G_{2}}$}
\put(164,-2.5){$\ll $}
%
%EndExpansion
$

 For more details on these classifications see \cite{humphreys}, \cite{kac}.
The following result   is useful in computing the characteristic polynomial of the Cartan matrix.

\begin{proposition} \label{th1}
Let $C$ be the $n \times n$  Cartan matrix of a simple Lie algebra over $\mathbf{ C}$.  Let $p_n(x)$ be its characteristic polynomial.  Then
\begin{displaymath}  p_n(x)=q_n \left(\frac{x}{2} -1 \right)   \end{displaymath} where $q_n$ is a polynomial related to Chebyshev polynomials as follows:

\begin{align*}
A_n:  \ \ \ \ \ &  \ \  q_n= U_n \\
B_n,  C_n:   &  \ \ q_n=2 T_n \\
D_n:  \ \ \ \ \  & \ \    q_n=4x T_{n-1}
\end{align*}

where $T_n$ and $U_n$ are the Chebyshev polynomials of first and second kind respectively.
\end{proposition}

To a given Dynkin diagram $\Gamma$ with $n$ nodes we  associate the \textit{ Coxeter adjacency matrix} which is  the $n\times n$  matrix $A=2I-C$  where  $C$ is the Cartan matrix. The \textit{characteristic polynomial} of $\Gamma$ is that of $A$. Similarly the \textit{norm} of $\Gamma$ is defined to be the norm of $A$. One  defines, see \cite{doobs},  the \textit{spectral radius} of $\Gamma $ to be
                  \bd  \rho (\Gamma)= max\{|\lambda|: \lambda  \ is \ an \  eigenvalue \ of \ A \} \ . \ed

 If the graph is a tree then  the characteristic polynomial $p_A$ of the adjacency matrix is simply related to the characteristic polynomial of the Cartan matrix $p_C$. In fact:

\be \label{ptoa}  a_n(x)=p_n (x+2) \ .  \ee

Using the fact that the spectrum  of $A$ is  the same as the spectrum of $-A$ it follows easily that if $\lambda $ is an eigenvalue of the adjacency matrix then $2+\lambda$ is an eigenvalue of the corresponding Cartan matrix.

In this paper we use the following notation.  Note that the subscript $n$,  in all cases,  is equal to the degree of the polynomial except in the case of  $Q_n(x)$ which is of degree $2n$.

\begin{itemize}
\item
$p_n(x)$ will denote the characteristic polynomial of the Cartan matrix.

  \item $a_n(x)$ will denote the characteristic polynomial of the adjacency matrix.   Note that
\bd a_n(x)=p_n(x+2)=q_n( \frac{x}{2})  \ . \ed

\item
Finally we define the associated polynomial
\bd Q_n(x)=x^n a_n (x + \frac{1}{x})  \ . \ed
\end{itemize}

$Q_n(x)$ turns out to be an even, reciprocal  polynomial of the form $Q_n(x)=f_n(x^2)$, (see \cite{moody}, \cite{Lenzing}). The polynomial $f_n$ is the so called Coxeter polynomial of the underlying graph. For the definition and spectral properties  of the Coxeter polynomial see, \cite{aCampo},   \cite{moody}, \cite{coleman}, \cite{coxeter}, \cite{howlett},  \cite{kostant}, \cite{lakatos}, \cite{ringel}, \cite{steinberg}, \cite{steko}.
The roots of $Q_n$ in the cases we consider  are in the unit disk and therefore by a theorem of Kronecker, see \cite{damianou},  $Q_n(x)$  is a product of cyclotomic polynomials. We determine the factorization of $f_n$ as a product of cyclotomic polynomials. This factorization in turn determines  the factorization of $Q_n$.   The irreducible factors of $Q_n$ are in one-to-one correspondence with the irreducible factors of $a_n(x)$.  As a bi-product we obtain the factorization of the Chebyshev polynomials of the first and second kind over the integers. More precisely we prove the following result:

\begin{theorem}
Let $\psi_n(x)$ be the minimal polynomial of the algebraic integer $2 \cos \frac{2 \pi}{n}$. Then

\bd U_n(x)=\prod_{\substack{ j|2n+2 \\  j\not=1,2}} \psi_j(2x) \ . \ed
Let $n=2^{\alpha} N$ where $N$ is odd and let $r=2^{\alpha+2}$. Then

\bd  T_n(x)=\frac{1}{2}\prod_{\substack{ j|N \\   }} \psi_{r j}(2x) \ . \ed
\end{theorem}

The irreducible polynomials $\psi_n$ were introduced by Lehmer in \cite{lehmer}. The factorization is consistent with previous results, e.g.  \cite{hsiao}, \cite{rayes}.

Using the factorization of the polynomials $a_n(x)$ and $p_n(x)$ we obtain the following interesting sine formula. See \cite{damianou3} for various appearances and interesting connections.

\begin{theorem}
Let  $\mathfrak{g}$ be a complex simple Lie algebra of rank $\ell$,  $h$ the Coxeter number,  $m_1, m_2, \dots, m_{\ell}$ the exponents of $\mathfrak{g}$,  and $C$ the Cartan matrix.  Then

\bd 2^{2 \ell} \prod_{i=1}^{\ell}  \sin^2 {\frac{ m_i \pi}{2h}     } ={\rm det}\ C  \ . \ed
\end{theorem}

\section{Chebyshev polynomials} \label{cp}

To compute explicitly $p_n(x)$, we use   the  following result,  see \cite{coxeter} and  \cite{koranyi}.
\begin{proposition} \label{main}
Let $C$ be the $n \times n$  Cartan matrix of a simple Lie algebra over $\mathbb{ C}$.  Let $p_n(x)$ be its characteristic polynomial and define $q_n(x)={\rm det} (2x I +A)$.  Then
\begin{displaymath}  p_n(x)=q_n \left(\frac{x}{2} -1 \right),  \ \ \  and \ \ \  a_n(x)=q_n \left( \frac{x}{2}\right)  \ . \end{displaymath}

The polynomial   $q_n$ is  related to Chebyshev polynomials as follows:

\begin{align*}
A_n:  \ \ \ \ \ &  \ \  q_n= U_n \\
B_n,  C_n:   &  \ \ q_n=2 T_n \\
D_n:  \ \ \ \ \  & \ \    q_n=4x T_{n-1}
\end{align*}
where $T_n$ and $U_n$ are the Chebyshev polynomials of first and second kind respectively.
\end{proposition}

\noindent \textit{Proof.} We give an outline of the proof. Note that

\begin{align*} q_n\left( \frac{x}{2}-1 \right) &={\rm det}\, \left( 2 \left( \frac{x-2}{2}\right)I_n +A \right)  \\
& = {\rm det}\, \left( xI_n   -2I_n +A \right) \\
& ={\rm det}\, \left( xI_n  -2I_n +2I_n-C \right) \\
&= {\rm det}\, \left( xI_n  -C \right) =p_n(x) \ . \end{align*}

Furthermore,  the matrix $A$ for classical Lie algebras has the form
\bd \label{cartan} A= \begin{pmatrix} 0 & 1&  & & &  \cr
                     1 &0 &1& &&  \cr
                       & \ddots &\ddots \ddots & & & \cr
                       & & \ddots & \ddots & \ddots&  \cr
                        & & & 1&0& 1  \cr
                        & & & & 1 & D\end{pmatrix} \ ,  \ed

                        where $D$ is
       \bd (0),  \begin{pmatrix} 0 & 2  \cr 1 & 0 \end{pmatrix}, \begin{pmatrix} 0 & 1  \cr 2 & 0 \end{pmatrix},  \begin{pmatrix} 0 &1 & 1 \cr 1 &0 & 0 \cr 1 &0&0  \end{pmatrix}  \ed
       for the cases $A_n$, $B_n$, $C_n$, $D_n$ respectively.   The proof is by induction on $n$.  Suppose the result is proved for $n$ and $n-1$.  To get $q_{n+1}(x)$, expand the determinant of $2xI_{n+1}+A$ by the first two rows to obtain
       \bd q_{n+1}(x) = 2x q_n(x) -q_{n-1} \ . \ed
       This is the recurrence relation satisfied by the Chebyshev polynomials $T_n$ and $U_n$.   Sections ~\ref{can} and ~\ref{cbn} cover in detail  the  cases $A_n$ and $B_n$ respectively. In section ~\ref{cdn}  we  outline the case of $D_n$.

       Finally, note  that
       \bd a_n(x)=p_n(x+2)=q_n \left( \frac{x+2}{2} -1  \right) =q_n \left( \frac{x}{2} \right)  \ . \hspace{4 cm}  \blacksquare\ed

 It  is clear from what follows in this paper,  that Chebyshev polynomials  were designed  in order to fit nicely the theory of  complex simple Lie algebras. The precise definition and some basic properties  of  $U_n$ and $T_n$ will be given in sections  ~\ref{can} and ~\ref{cbn} respectively. In the case of exceptional Lie algebras, one can directly compute (preferably using   a symbolic manipulation package in the case of $E_n$) the characteristic polynomials $a_n(x)$ and $p_n(x)$ for each exceptional type. We present the calculations in section ~\ref{cex}.

\section{Cartan matrix of type $A_n$ } \label{can}

\subsection{Eigenvalues of the $A_n$ Cartan matrix}
Toeplitz matrices  have constant entries on each diagonal parallel to the main diagonal. Tridiagonal Toeplitz matrices are commonly the result of discretizing differential equations.

The eigenvalues of the Toeplitz matrix

\be \label{top} \begin{pmatrix} b & a&  & & &  \cr
                     c &b &a & &&  \cr
                       & \ddots &\ddots \ddots & & & \cr
                       & & \ddots & \ddots & \ddots&  \cr
                        & & & c&b& a  \cr
                        & & & & c & b \end{pmatrix}  \ee
                        are given by
\be \label{topeigen}
\lambda_j=b+2a \sqrt{   \frac{c}{a}        } \cos{   \frac{j \pi}{n+1}     }  \ \ \ j=1, 2 \dots, n \ ,
\ee
see e.g. \cite[p. 59]{smith}.

The Cartan matrix of type $A_n$ is a tri-diagonal matrix of the form.
\be \label{cartan} C_{A_n}= \begin{pmatrix} 2 & -1&  & & &  \cr
                     -1 &2 &-1& &&  \cr
                       & \ddots &\ddots \ddots & & & \cr
                       & & \ddots & \ddots & \ddots&  \cr
                        & & & -1&2& -1  \cr
                        & & & & -1 & 2\end{pmatrix} \ .  \ee
It appears in the classification theory of simple Lie algebras over $\mathbf{C}$.

Taking  $a=c=-1$, $b=2$ in (\ref{topeigen})  we deduce that the eigenvalues of $A_n$ are given by
\be \label{eigen} \lambda_j=2 -2 \cos {  \frac{ j \pi}{n+1} } =4 \sin^2 {  \frac{ j \pi}{2(n+1)} }  \ \ \ j=1,2, \dots, n  \ .  \ee

Let $d_n$ be the determinant of $C_{A_n}$.  One  can compute it using expansion on the first row and induction.  We  obtain
$d_n=2 d_{n-1} - d_{n-2}$,  $d_1=2$, $d_2=3$. This is a simple linear recurrence with solution $d_n=n+1$.

 We conclude that

\bd \prod_{j=1}^n 4 \sin^2 {  \frac{ j \pi}{2(n+1)} } = n+1   \ . \ed

Equivalently

\be  \label{det}    2^{2n} \prod_{j=1}^n  \sin^2 {  \frac{ j \pi}{2(n+1)} } = n+1  \ .  \ \ \ \ \ \ \ \ \ \ \left( A_n \ \ \   {\rm sine \ \ \  formula} \right) \ee

We refer to this relation as the $A_n$ sine formula.

\subsection{Product of distances from $1$ to the other roots of unity}
\vskip 0.5 cm
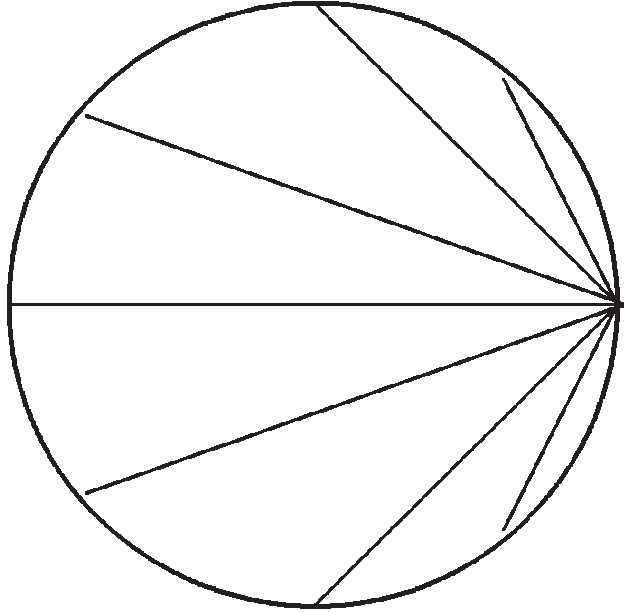
\begin{center}
\ifx\JPicScale\undefined\def\JPicScale{1}\fi
\unitlength \JPicScale mm
\begin{picture}(90,90)(0,0)
\linethickness{0.5mm}
\put(90,49.75){\line(0,1){0.5}}
\multiput(89.99,50.75)(0.01,-0.5){1}{\line(0,-1){0.5}}
\multiput(89.98,51.25)(0.01,-0.5){1}{\line(0,-1){0.5}}
\multiput(89.96,51.74)(0.02,-0.5){1}{\line(0,-1){0.5}}
\multiput(89.94,52.24)(0.02,-0.5){1}{\line(0,-1){0.5}}
\multiput(89.91,52.74)(0.03,-0.5){1}{\line(0,-1){0.5}}
\multiput(89.87,53.24)(0.04,-0.5){1}{\line(0,-1){0.5}}
\multiput(89.83,53.73)(0.04,-0.5){1}{\line(0,-1){0.5}}
\multiput(89.78,54.23)(0.05,-0.5){1}{\line(0,-1){0.5}}
\multiput(89.72,54.73)(0.06,-0.5){1}{\line(0,-1){0.5}}
\multiput(89.66,55.22)(0.06,-0.49){1}{\line(0,-1){0.49}}
\multiput(89.59,55.72)(0.07,-0.49){1}{\line(0,-1){0.49}}
\multiput(89.52,56.21)(0.07,-0.49){1}{\line(0,-1){0.49}}
\multiput(89.43,56.7)(0.08,-0.49){1}{\line(0,-1){0.49}}
\multiput(89.35,57.19)(0.09,-0.49){1}{\line(0,-1){0.49}}
\multiput(89.26,57.68)(0.09,-0.49){1}{\line(0,-1){0.49}}
\multiput(89.16,58.17)(0.1,-0.49){1}{\line(0,-1){0.49}}
\multiput(89.05,58.66)(0.1,-0.49){1}{\line(0,-1){0.49}}
\multiput(88.94,59.14)(0.11,-0.49){1}{\line(0,-1){0.49}}
\multiput(88.82,59.63)(0.12,-0.48){1}{\line(0,-1){0.48}}
\multiput(88.7,60.11)(0.12,-0.48){1}{\line(0,-1){0.48}}
\multiput(88.57,60.59)(0.13,-0.48){1}{\line(0,-1){0.48}}
\multiput(88.44,61.07)(0.14,-0.48){1}{\line(0,-1){0.48}}
\multiput(88.3,61.55)(0.14,-0.48){1}{\line(0,-1){0.48}}
\multiput(88.15,62.03)(0.15,-0.48){1}{\line(0,-1){0.48}}
\multiput(88,62.5)(0.15,-0.47){1}{\line(0,-1){0.47}}
\multiput(87.84,62.98)(0.16,-0.47){1}{\line(0,-1){0.47}}
\multiput(87.67,63.45)(0.16,-0.47){1}{\line(0,-1){0.47}}
\multiput(87.5,63.91)(0.17,-0.47){1}{\line(0,-1){0.47}}
\multiput(87.33,64.38)(0.18,-0.47){1}{\line(0,-1){0.47}}
\multiput(87.14,64.85)(0.09,-0.23){2}{\line(0,-1){0.23}}
\multiput(86.96,65.31)(0.09,-0.23){2}{\line(0,-1){0.23}}
\multiput(86.76,65.77)(0.1,-0.23){2}{\line(0,-1){0.23}}
\multiput(86.56,66.22)(0.1,-0.23){2}{\line(0,-1){0.23}}
\multiput(86.36,66.68)(0.1,-0.23){2}{\line(0,-1){0.23}}
\multiput(86.15,67.13)(0.11,-0.23){2}{\line(0,-1){0.23}}
\multiput(85.93,67.58)(0.11,-0.22){2}{\line(0,-1){0.22}}
\multiput(85.71,68.03)(0.11,-0.22){2}{\line(0,-1){0.22}}
\multiput(85.48,68.47)(0.11,-0.22){2}{\line(0,-1){0.22}}
\multiput(85.25,68.91)(0.12,-0.22){2}{\line(0,-1){0.22}}
\multiput(85.01,69.35)(0.12,-0.22){2}{\line(0,-1){0.22}}
\multiput(84.77,69.78)(0.12,-0.22){2}{\line(0,-1){0.22}}
\multiput(84.52,70.22)(0.12,-0.22){2}{\line(0,-1){0.22}}
\multiput(84.26,70.64)(0.13,-0.21){2}{\line(0,-1){0.21}}
\multiput(84,71.07)(0.13,-0.21){2}{\line(0,-1){0.21}}
\multiput(83.74,71.49)(0.13,-0.21){2}{\line(0,-1){0.21}}
\multiput(83.47,71.91)(0.14,-0.21){2}{\line(0,-1){0.21}}
\multiput(83.19,72.33)(0.14,-0.21){2}{\line(0,-1){0.21}}
\multiput(82.91,72.74)(0.14,-0.21){2}{\line(0,-1){0.21}}
\multiput(82.62,73.15)(0.14,-0.2){2}{\line(0,-1){0.2}}
\multiput(82.33,73.55)(0.15,-0.2){2}{\line(0,-1){0.2}}
\multiput(82.04,73.95)(0.15,-0.2){2}{\line(0,-1){0.2}}
\multiput(81.73,74.35)(0.1,-0.13){3}{\line(0,-1){0.13}}
\multiput(81.43,74.74)(0.1,-0.13){3}{\line(0,-1){0.13}}
\multiput(81.12,75.13)(0.1,-0.13){3}{\line(0,-1){0.13}}
\multiput(80.8,75.52)(0.11,-0.13){3}{\line(0,-1){0.13}}
\multiput(80.48,75.9)(0.11,-0.13){3}{\line(0,-1){0.13}}
\multiput(80.16,76.28)(0.11,-0.13){3}{\line(0,-1){0.13}}
\multiput(79.83,76.65)(0.11,-0.12){3}{\line(0,-1){0.12}}
\multiput(79.49,77.02)(0.11,-0.12){3}{\line(0,-1){0.12}}
\multiput(79.15,77.39)(0.11,-0.12){3}{\line(0,-1){0.12}}
\multiput(78.81,77.75)(0.11,-0.12){3}{\line(0,-1){0.12}}
\multiput(78.46,78.11)(0.12,-0.12){3}{\line(0,-1){0.12}}
\multiput(78.11,78.46)(0.12,-0.12){3}{\line(1,0){0.12}}
\multiput(77.75,78.81)(0.12,-0.12){3}{\line(1,0){0.12}}
\multiput(77.39,79.15)(0.12,-0.11){3}{\line(1,0){0.12}}
\multiput(77.02,79.49)(0.12,-0.11){3}{\line(1,0){0.12}}
\multiput(76.65,79.83)(0.12,-0.11){3}{\line(1,0){0.12}}
\multiput(76.28,80.16)(0.12,-0.11){3}{\line(1,0){0.12}}
\multiput(75.9,80.48)(0.13,-0.11){3}{\line(1,0){0.13}}
\multiput(75.52,80.8)(0.13,-0.11){3}{\line(1,0){0.13}}
\multiput(75.13,81.12)(0.13,-0.11){3}{\line(1,0){0.13}}
\multiput(74.74,81.43)(0.13,-0.1){3}{\line(1,0){0.13}}
\multiput(74.35,81.73)(0.13,-0.1){3}{\line(1,0){0.13}}
\multiput(73.95,82.04)(0.13,-0.1){3}{\line(1,0){0.13}}
\multiput(73.55,82.33)(0.2,-0.15){2}{\line(1,0){0.2}}
\multiput(73.15,82.62)(0.2,-0.15){2}{\line(1,0){0.2}}
\multiput(72.74,82.91)(0.2,-0.14){2}{\line(1,0){0.2}}
\multiput(72.33,83.19)(0.21,-0.14){2}{\line(1,0){0.21}}
\multiput(71.91,83.47)(0.21,-0.14){2}{\line(1,0){0.21}}
\multiput(71.49,83.74)(0.21,-0.14){2}{\line(1,0){0.21}}
\multiput(71.07,84)(0.21,-0.13){2}{\line(1,0){0.21}}
\multiput(70.64,84.26)(0.21,-0.13){2}{\line(1,0){0.21}}
\multiput(70.22,84.52)(0.21,-0.13){2}{\line(1,0){0.21}}
\multiput(69.78,84.77)(0.22,-0.12){2}{\line(1,0){0.22}}
\multiput(69.35,85.01)(0.22,-0.12){2}{\line(1,0){0.22}}
\multiput(68.91,85.25)(0.22,-0.12){2}{\line(1,0){0.22}}
\multiput(68.47,85.48)(0.22,-0.12){2}{\line(1,0){0.22}}
\multiput(68.03,85.71)(0.22,-0.11){2}{\line(1,0){0.22}}
\multiput(67.58,85.93)(0.22,-0.11){2}{\line(1,0){0.22}}
\multiput(67.13,86.15)(0.22,-0.11){2}{\line(1,0){0.22}}
\multiput(66.68,86.36)(0.23,-0.11){2}{\line(1,0){0.23}}
\multiput(66.22,86.56)(0.23,-0.1){2}{\line(1,0){0.23}}
\multiput(65.77,86.76)(0.23,-0.1){2}{\line(1,0){0.23}}
\multiput(65.31,86.96)(0.23,-0.1){2}{\line(1,0){0.23}}
\multiput(64.85,87.14)(0.23,-0.09){2}{\line(1,0){0.23}}
\multiput(64.38,87.33)(0.23,-0.09){2}{\line(1,0){0.23}}
\multiput(63.91,87.5)(0.47,-0.18){1}{\line(1,0){0.47}}
\multiput(63.45,87.67)(0.47,-0.17){1}{\line(1,0){0.47}}
\multiput(62.98,87.84)(0.47,-0.16){1}{\line(1,0){0.47}}
\multiput(62.5,88)(0.47,-0.16){1}{\line(1,0){0.47}}
\multiput(62.03,88.15)(0.47,-0.15){1}{\line(1,0){0.47}}
\multiput(61.55,88.3)(0.48,-0.15){1}{\line(1,0){0.48}}
\multiput(61.07,88.44)(0.48,-0.14){1}{\line(1,0){0.48}}
\multiput(60.59,88.57)(0.48,-0.14){1}{\line(1,0){0.48}}
\multiput(60.11,88.7)(0.48,-0.13){1}{\line(1,0){0.48}}
\multiput(59.63,88.82)(0.48,-0.12){1}{\line(1,0){0.48}}
\multiput(59.14,88.94)(0.48,-0.12){1}{\line(1,0){0.48}}
\multiput(58.66,89.05)(0.49,-0.11){1}{\line(1,0){0.49}}
\multiput(58.17,89.16)(0.49,-0.1){1}{\line(1,0){0.49}}
\multiput(57.68,89.26)(0.49,-0.1){1}{\line(1,0){0.49}}
\multiput(57.19,89.35)(0.49,-0.09){1}{\line(1,0){0.49}}
\multiput(56.7,89.43)(0.49,-0.09){1}{\line(1,0){0.49}}
\multiput(56.21,89.52)(0.49,-0.08){1}{\line(1,0){0.49}}
\multiput(55.72,89.59)(0.49,-0.07){1}{\line(1,0){0.49}}
\multiput(55.22,89.66)(0.49,-0.07){1}{\line(1,0){0.49}}
\multiput(54.73,89.72)(0.49,-0.06){1}{\line(1,0){0.49}}
\multiput(54.23,89.78)(0.5,-0.06){1}{\line(1,0){0.5}}
\multiput(53.73,89.83)(0.5,-0.05){1}{\line(1,0){0.5}}
\multiput(53.24,89.87)(0.5,-0.04){1}{\line(1,0){0.5}}
\multiput(52.74,89.91)(0.5,-0.04){1}{\line(1,0){0.5}}
\multiput(52.24,89.94)(0.5,-0.03){1}{\line(1,0){0.5}}
\multiput(51.74,89.96)(0.5,-0.02){1}{\line(1,0){0.5}}
\multiput(51.25,89.98)(0.5,-0.02){1}{\line(1,0){0.5}}
\multiput(50.75,89.99)(0.5,-0.01){1}{\line(1,0){0.5}}
\multiput(50.25,90)(0.5,-0.01){1}{\line(1,0){0.5}}
\put(49.75,90){\line(1,0){0.5}}
\multiput(49.25,89.99)(0.5,0.01){1}{\line(1,0){0.5}}
\multiput(48.75,89.98)(0.5,0.01){1}{\line(1,0){0.5}}
\multiput(48.26,89.96)(0.5,0.02){1}{\line(1,0){0.5}}
\multiput(47.76,89.94)(0.5,0.02){1}{\line(1,0){0.5}}
\multiput(47.26,89.91)(0.5,0.03){1}{\line(1,0){0.5}}
\multiput(46.76,89.87)(0.5,0.04){1}{\line(1,0){0.5}}
\multiput(46.27,89.83)(0.5,0.04){1}{\line(1,0){0.5}}
\multiput(45.77,89.78)(0.5,0.05){1}{\line(1,0){0.5}}
\multiput(45.27,89.72)(0.5,0.06){1}{\line(1,0){0.5}}
\multiput(44.78,89.66)(0.49,0.06){1}{\line(1,0){0.49}}
\multiput(44.28,89.59)(0.49,0.07){1}{\line(1,0){0.49}}
\multiput(43.79,89.52)(0.49,0.07){1}{\line(1,0){0.49}}
\multiput(43.3,89.43)(0.49,0.08){1}{\line(1,0){0.49}}
\multiput(42.81,89.35)(0.49,0.09){1}{\line(1,0){0.49}}
\multiput(42.32,89.26)(0.49,0.09){1}{\line(1,0){0.49}}
\multiput(41.83,89.16)(0.49,0.1){1}{\line(1,0){0.49}}
\multiput(41.34,89.05)(0.49,0.1){1}{\line(1,0){0.49}}
\multiput(40.86,88.94)(0.49,0.11){1}{\line(1,0){0.49}}
\multiput(40.37,88.82)(0.48,0.12){1}{\line(1,0){0.48}}
\multiput(39.89,88.7)(0.48,0.12){1}{\line(1,0){0.48}}
\multiput(39.41,88.57)(0.48,0.13){1}{\line(1,0){0.48}}
\multiput(38.93,88.44)(0.48,0.14){1}{\line(1,0){0.48}}
\multiput(38.45,88.3)(0.48,0.14){1}{\line(1,0){0.48}}
\multiput(37.97,88.15)(0.48,0.15){1}{\line(1,0){0.48}}
\multiput(37.5,88)(0.47,0.15){1}{\line(1,0){0.47}}
\multiput(37.02,87.84)(0.47,0.16){1}{\line(1,0){0.47}}
\multiput(36.55,87.67)(0.47,0.16){1}{\line(1,0){0.47}}
\multiput(36.09,87.5)(0.47,0.17){1}{\line(1,0){0.47}}
\multiput(35.62,87.33)(0.47,0.18){1}{\line(1,0){0.47}}
\multiput(35.15,87.14)(0.23,0.09){2}{\line(1,0){0.23}}
\multiput(34.69,86.96)(0.23,0.09){2}{\line(1,0){0.23}}
\multiput(34.23,86.76)(0.23,0.1){2}{\line(1,0){0.23}}
\multiput(33.78,86.56)(0.23,0.1){2}{\line(1,0){0.23}}
\multiput(33.32,86.36)(0.23,0.1){2}{\line(1,0){0.23}}
\multiput(32.87,86.15)(0.23,0.11){2}{\line(1,0){0.23}}
\multiput(32.42,85.93)(0.22,0.11){2}{\line(1,0){0.22}}
\multiput(31.97,85.71)(0.22,0.11){2}{\line(1,0){0.22}}
\multiput(31.53,85.48)(0.22,0.11){2}{\line(1,0){0.22}}
\multiput(31.09,85.25)(0.22,0.12){2}{\line(1,0){0.22}}
\multiput(30.65,85.01)(0.22,0.12){2}{\line(1,0){0.22}}
\multiput(30.22,84.77)(0.22,0.12){2}{\line(1,0){0.22}}
\multiput(29.78,84.52)(0.22,0.12){2}{\line(1,0){0.22}}
\multiput(29.36,84.26)(0.21,0.13){2}{\line(1,0){0.21}}
\multiput(28.93,84)(0.21,0.13){2}{\line(1,0){0.21}}
\multiput(28.51,83.74)(0.21,0.13){2}{\line(1,0){0.21}}
\multiput(28.09,83.47)(0.21,0.14){2}{\line(1,0){0.21}}
\multiput(27.67,83.19)(0.21,0.14){2}{\line(1,0){0.21}}
\multiput(27.26,82.91)(0.21,0.14){2}{\line(1,0){0.21}}
\multiput(26.85,82.62)(0.2,0.14){2}{\line(1,0){0.2}}
\multiput(26.45,82.33)(0.2,0.15){2}{\line(1,0){0.2}}
\multiput(26.05,82.04)(0.2,0.15){2}{\line(1,0){0.2}}
\multiput(25.65,81.73)(0.13,0.1){3}{\line(1,0){0.13}}
\multiput(25.26,81.43)(0.13,0.1){3}{\line(1,0){0.13}}
\multiput(24.87,81.12)(0.13,0.1){3}{\line(1,0){0.13}}
\multiput(24.48,80.8)(0.13,0.11){3}{\line(1,0){0.13}}
\multiput(24.1,80.48)(0.13,0.11){3}{\line(1,0){0.13}}
\multiput(23.72,80.16)(0.13,0.11){3}{\line(1,0){0.13}}
\multiput(23.35,79.83)(0.12,0.11){3}{\line(1,0){0.12}}
\multiput(22.98,79.49)(0.12,0.11){3}{\line(1,0){0.12}}
\multiput(22.61,79.15)(0.12,0.11){3}{\line(1,0){0.12}}
\multiput(22.25,78.81)(0.12,0.11){3}{\line(1,0){0.12}}
\multiput(21.89,78.46)(0.12,0.12){3}{\line(1,0){0.12}}
\multiput(21.54,78.11)(0.12,0.12){3}{\line(1,0){0.12}}
\multiput(21.19,77.75)(0.12,0.12){3}{\line(0,1){0.12}}
\multiput(20.85,77.39)(0.11,0.12){3}{\line(0,1){0.12}}
\multiput(20.51,77.02)(0.11,0.12){3}{\line(0,1){0.12}}
\multiput(20.17,76.65)(0.11,0.12){3}{\line(0,1){0.12}}
\multiput(19.84,76.28)(0.11,0.12){3}{\line(0,1){0.12}}
\multiput(19.52,75.9)(0.11,0.13){3}{\line(0,1){0.13}}
\multiput(19.2,75.52)(0.11,0.13){3}{\line(0,1){0.13}}
\multiput(18.88,75.13)(0.11,0.13){3}{\line(0,1){0.13}}
\multiput(18.57,74.74)(0.1,0.13){3}{\line(0,1){0.13}}
\multiput(18.27,74.35)(0.1,0.13){3}{\line(0,1){0.13}}
\multiput(17.96,73.95)(0.1,0.13){3}{\line(0,1){0.13}}
\multiput(17.67,73.55)(0.15,0.2){2}{\line(0,1){0.2}}
\multiput(17.38,73.15)(0.15,0.2){2}{\line(0,1){0.2}}
\multiput(17.09,72.74)(0.14,0.2){2}{\line(0,1){0.2}}
\multiput(16.81,72.33)(0.14,0.21){2}{\line(0,1){0.21}}
\multiput(16.53,71.91)(0.14,0.21){2}{\line(0,1){0.21}}
\multiput(16.26,71.49)(0.14,0.21){2}{\line(0,1){0.21}}
\multiput(16,71.07)(0.13,0.21){2}{\line(0,1){0.21}}
\multiput(15.74,70.64)(0.13,0.21){2}{\line(0,1){0.21}}
\multiput(15.48,70.22)(0.13,0.21){2}{\line(0,1){0.21}}
\multiput(15.23,69.78)(0.12,0.22){2}{\line(0,1){0.22}}
\multiput(14.99,69.35)(0.12,0.22){2}{\line(0,1){0.22}}
\multiput(14.75,68.91)(0.12,0.22){2}{\line(0,1){0.22}}
\multiput(14.52,68.47)(0.12,0.22){2}{\line(0,1){0.22}}
\multiput(14.29,68.03)(0.11,0.22){2}{\line(0,1){0.22}}
\multiput(14.07,67.58)(0.11,0.22){2}{\line(0,1){0.22}}
\multiput(13.85,67.13)(0.11,0.22){2}{\line(0,1){0.22}}
\multiput(13.64,66.68)(0.11,0.23){2}{\line(0,1){0.23}}
\multiput(13.44,66.22)(0.1,0.23){2}{\line(0,1){0.23}}
\multiput(13.24,65.77)(0.1,0.23){2}{\line(0,1){0.23}}
\multiput(13.04,65.31)(0.1,0.23){2}{\line(0,1){0.23}}
\multiput(12.86,64.85)(0.09,0.23){2}{\line(0,1){0.23}}
\multiput(12.67,64.38)(0.09,0.23){2}{\line(0,1){0.23}}
\multiput(12.5,63.91)(0.18,0.47){1}{\line(0,1){0.47}}
\multiput(12.33,63.45)(0.17,0.47){1}{\line(0,1){0.47}}
\multiput(12.16,62.98)(0.16,0.47){1}{\line(0,1){0.47}}
\multiput(12,62.5)(0.16,0.47){1}{\line(0,1){0.47}}
\multiput(11.85,62.03)(0.15,0.47){1}{\line(0,1){0.47}}
\multiput(11.7,61.55)(0.15,0.48){1}{\line(0,1){0.48}}
\multiput(11.56,61.07)(0.14,0.48){1}{\line(0,1){0.48}}
\multiput(11.43,60.59)(0.14,0.48){1}{\line(0,1){0.48}}
\multiput(11.3,60.11)(0.13,0.48){1}{\line(0,1){0.48}}
\multiput(11.18,59.63)(0.12,0.48){1}{\line(0,1){0.48}}
\multiput(11.06,59.14)(0.12,0.48){1}{\line(0,1){0.48}}
\multiput(10.95,58.66)(0.11,0.49){1}{\line(0,1){0.49}}
\multiput(10.84,58.17)(0.1,0.49){1}{\line(0,1){0.49}}
\multiput(10.74,57.68)(0.1,0.49){1}{\line(0,1){0.49}}
\multiput(10.65,57.19)(0.09,0.49){1}{\line(0,1){0.49}}
\multiput(10.57,56.7)(0.09,0.49){1}{\line(0,1){0.49}}
\multiput(10.48,56.21)(0.08,0.49){1}{\line(0,1){0.49}}
\multiput(10.41,55.72)(0.07,0.49){1}{\line(0,1){0.49}}
\multiput(10.34,55.22)(0.07,0.49){1}{\line(0,1){0.49}}
\multiput(10.28,54.73)(0.06,0.49){1}{\line(0,1){0.49}}
\multiput(10.22,54.23)(0.06,0.5){1}{\line(0,1){0.5}}
\multiput(10.17,53.73)(0.05,0.5){1}{\line(0,1){0.5}}
\multiput(10.13,53.24)(0.04,0.5){1}{\line(0,1){0.5}}
\multiput(10.09,52.74)(0.04,0.5){1}{\line(0,1){0.5}}
\multiput(10.06,52.24)(0.03,0.5){1}{\line(0,1){0.5}}
\multiput(10.04,51.74)(0.02,0.5){1}{\line(0,1){0.5}}
\multiput(10.02,51.25)(0.02,0.5){1}{\line(0,1){0.5}}
\multiput(10.01,50.75)(0.01,0.5){1}{\line(0,1){0.5}}
\multiput(10,50.25)(0.01,0.5){1}{\line(0,1){0.5}}
\put(10,49.75){\line(0,1){0.5}}
\multiput(10,49.75)(0.01,-0.5){1}{\line(0,-1){0.5}}
\multiput(10.01,49.25)(0.01,-0.5){1}{\line(0,-1){0.5}}
\multiput(10.02,48.75)(0.02,-0.5){1}{\line(0,-1){0.5}}
\multiput(10.04,48.26)(0.02,-0.5){1}{\line(0,-1){0.5}}
\multiput(10.06,47.76)(0.03,-0.5){1}{\line(0,-1){0.5}}
\multiput(10.09,47.26)(0.04,-0.5){1}{\line(0,-1){0.5}}
\multiput(10.13,46.76)(0.04,-0.5){1}{\line(0,-1){0.5}}
\multiput(10.17,46.27)(0.05,-0.5){1}{\line(0,-1){0.5}}
\multiput(10.22,45.77)(0.06,-0.5){1}{\line(0,-1){0.5}}
\multiput(10.28,45.27)(0.06,-0.49){1}{\line(0,-1){0.49}}
\multiput(10.34,44.78)(0.07,-0.49){1}{\line(0,-1){0.49}}
\multiput(10.41,44.28)(0.07,-0.49){1}{\line(0,-1){0.49}}
\multiput(10.48,43.79)(0.08,-0.49){1}{\line(0,-1){0.49}}
\multiput(10.57,43.3)(0.09,-0.49){1}{\line(0,-1){0.49}}
\multiput(10.65,42.81)(0.09,-0.49){1}{\line(0,-1){0.49}}
\multiput(10.74,42.32)(0.1,-0.49){1}{\line(0,-1){0.49}}
\multiput(10.84,41.83)(0.1,-0.49){1}{\line(0,-1){0.49}}
\multiput(10.95,41.34)(0.11,-0.49){1}{\line(0,-1){0.49}}
\multiput(11.06,40.86)(0.12,-0.48){1}{\line(0,-1){0.48}}
\multiput(11.18,40.37)(0.12,-0.48){1}{\line(0,-1){0.48}}
\multiput(11.3,39.89)(0.13,-0.48){1}{\line(0,-1){0.48}}
\multiput(11.43,39.41)(0.14,-0.48){1}{\line(0,-1){0.48}}
\multiput(11.56,38.93)(0.14,-0.48){1}{\line(0,-1){0.48}}
\multiput(11.7,38.45)(0.15,-0.48){1}{\line(0,-1){0.48}}
\multiput(11.85,37.97)(0.15,-0.47){1}{\line(0,-1){0.47}}
\multiput(12,37.5)(0.16,-0.47){1}{\line(0,-1){0.47}}
\multiput(12.16,37.02)(0.16,-0.47){1}{\line(0,-1){0.47}}
\multiput(12.33,36.55)(0.17,-0.47){1}{\line(0,-1){0.47}}
\multiput(12.5,36.09)(0.18,-0.47){1}{\line(0,-1){0.47}}
\multiput(12.67,35.62)(0.09,-0.23){2}{\line(0,-1){0.23}}
\multiput(12.86,35.15)(0.09,-0.23){2}{\line(0,-1){0.23}}
\multiput(13.04,34.69)(0.1,-0.23){2}{\line(0,-1){0.23}}
\multiput(13.24,34.23)(0.1,-0.23){2}{\line(0,-1){0.23}}
\multiput(13.44,33.78)(0.1,-0.23){2}{\line(0,-1){0.23}}
\multiput(13.64,33.32)(0.11,-0.23){2}{\line(0,-1){0.23}}
\multiput(13.85,32.87)(0.11,-0.22){2}{\line(0,-1){0.22}}
\multiput(14.07,32.42)(0.11,-0.22){2}{\line(0,-1){0.22}}
\multiput(14.29,31.97)(0.11,-0.22){2}{\line(0,-1){0.22}}
\multiput(14.52,31.53)(0.12,-0.22){2}{\line(0,-1){0.22}}
\multiput(14.75,31.09)(0.12,-0.22){2}{\line(0,-1){0.22}}
\multiput(14.99,30.65)(0.12,-0.22){2}{\line(0,-1){0.22}}
\multiput(15.23,30.22)(0.12,-0.22){2}{\line(0,-1){0.22}}
\multiput(15.48,29.78)(0.13,-0.21){2}{\line(0,-1){0.21}}
\multiput(15.74,29.36)(0.13,-0.21){2}{\line(0,-1){0.21}}
\multiput(16,28.93)(0.13,-0.21){2}{\line(0,-1){0.21}}
\multiput(16.26,28.51)(0.14,-0.21){2}{\line(0,-1){0.21}}
\multiput(16.53,28.09)(0.14,-0.21){2}{\line(0,-1){0.21}}
\multiput(16.81,27.67)(0.14,-0.21){2}{\line(0,-1){0.21}}
\multiput(17.09,27.26)(0.14,-0.2){2}{\line(0,-1){0.2}}
\multiput(17.38,26.85)(0.15,-0.2){2}{\line(0,-1){0.2}}
\multiput(17.67,26.45)(0.15,-0.2){2}{\line(0,-1){0.2}}
\multiput(17.96,26.05)(0.1,-0.13){3}{\line(0,-1){0.13}}
\multiput(18.27,25.65)(0.1,-0.13){3}{\line(0,-1){0.13}}
\multiput(18.57,25.26)(0.1,-0.13){3}{\line(0,-1){0.13}}
\multiput(18.88,24.87)(0.11,-0.13){3}{\line(0,-1){0.13}}
\multiput(19.2,24.48)(0.11,-0.13){3}{\line(0,-1){0.13}}
\multiput(19.52,24.1)(0.11,-0.13){3}{\line(0,-1){0.13}}
\multiput(19.84,23.72)(0.11,-0.12){3}{\line(0,-1){0.12}}
\multiput(20.17,23.35)(0.11,-0.12){3}{\line(0,-1){0.12}}
\multiput(20.51,22.98)(0.11,-0.12){3}{\line(0,-1){0.12}}
\multiput(20.85,22.61)(0.11,-0.12){3}{\line(0,-1){0.12}}
\multiput(21.19,22.25)(0.12,-0.12){3}{\line(0,-1){0.12}}
\multiput(21.54,21.89)(0.12,-0.12){3}{\line(1,0){0.12}}
\multiput(21.89,21.54)(0.12,-0.12){3}{\line(1,0){0.12}}
\multiput(22.25,21.19)(0.12,-0.11){3}{\line(1,0){0.12}}
\multiput(22.61,20.85)(0.12,-0.11){3}{\line(1,0){0.12}}
\multiput(22.98,20.51)(0.12,-0.11){3}{\line(1,0){0.12}}
\multiput(23.35,20.17)(0.12,-0.11){3}{\line(1,0){0.12}}
\multiput(23.72,19.84)(0.13,-0.11){3}{\line(1,0){0.13}}
\multiput(24.1,19.52)(0.13,-0.11){3}{\line(1,0){0.13}}
\multiput(24.48,19.2)(0.13,-0.11){3}{\line(1,0){0.13}}
\multiput(24.87,18.88)(0.13,-0.1){3}{\line(1,0){0.13}}
\multiput(25.26,18.57)(0.13,-0.1){3}{\line(1,0){0.13}}
\multiput(25.65,18.27)(0.13,-0.1){3}{\line(1,0){0.13}}
\multiput(26.05,17.96)(0.2,-0.15){2}{\line(1,0){0.2}}
\multiput(26.45,17.67)(0.2,-0.15){2}{\line(1,0){0.2}}
\multiput(26.85,17.38)(0.2,-0.14){2}{\line(1,0){0.2}}
\multiput(27.26,17.09)(0.21,-0.14){2}{\line(1,0){0.21}}
\multiput(27.67,16.81)(0.21,-0.14){2}{\line(1,0){0.21}}
\multiput(28.09,16.53)(0.21,-0.14){2}{\line(1,0){0.21}}
\multiput(28.51,16.26)(0.21,-0.13){2}{\line(1,0){0.21}}
\multiput(28.93,16)(0.21,-0.13){2}{\line(1,0){0.21}}
\multiput(29.36,15.74)(0.21,-0.13){2}{\line(1,0){0.21}}
\multiput(29.78,15.48)(0.22,-0.12){2}{\line(1,0){0.22}}
\multiput(30.22,15.23)(0.22,-0.12){2}{\line(1,0){0.22}}
\multiput(30.65,14.99)(0.22,-0.12){2}{\line(1,0){0.22}}
\multiput(31.09,14.75)(0.22,-0.12){2}{\line(1,0){0.22}}
\multiput(31.53,14.52)(0.22,-0.11){2}{\line(1,0){0.22}}
\multiput(31.97,14.29)(0.22,-0.11){2}{\line(1,0){0.22}}
\multiput(32.42,14.07)(0.22,-0.11){2}{\line(1,0){0.22}}
\multiput(32.87,13.85)(0.23,-0.11){2}{\line(1,0){0.23}}
\multiput(33.32,13.64)(0.23,-0.1){2}{\line(1,0){0.23}}
\multiput(33.78,13.44)(0.23,-0.1){2}{\line(1,0){0.23}}
\multiput(34.23,13.24)(0.23,-0.1){2}{\line(1,0){0.23}}
\multiput(34.69,13.04)(0.23,-0.09){2}{\line(1,0){0.23}}
\multiput(35.15,12.86)(0.23,-0.09){2}{\line(1,0){0.23}}
\multiput(35.62,12.67)(0.47,-0.18){1}{\line(1,0){0.47}}
\multiput(36.09,12.5)(0.47,-0.17){1}{\line(1,0){0.47}}
\multiput(36.55,12.33)(0.47,-0.16){1}{\line(1,0){0.47}}
\multiput(37.02,12.16)(0.47,-0.16){1}{\line(1,0){0.47}}
\multiput(37.5,12)(0.47,-0.15){1}{\line(1,0){0.47}}
\multiput(37.97,11.85)(0.48,-0.15){1}{\line(1,0){0.48}}
\multiput(38.45,11.7)(0.48,-0.14){1}{\line(1,0){0.48}}
\multiput(38.93,11.56)(0.48,-0.14){1}{\line(1,0){0.48}}
\multiput(39.41,11.43)(0.48,-0.13){1}{\line(1,0){0.48}}
\multiput(39.89,11.3)(0.48,-0.12){1}{\line(1,0){0.48}}
\multiput(40.37,11.18)(0.48,-0.12){1}{\line(1,0){0.48}}
\multiput(40.86,11.06)(0.49,-0.11){1}{\line(1,0){0.49}}
\multiput(41.34,10.95)(0.49,-0.1){1}{\line(1,0){0.49}}
\multiput(41.83,10.84)(0.49,-0.1){1}{\line(1,0){0.49}}
\multiput(42.32,10.74)(0.49,-0.09){1}{\line(1,0){0.49}}
\multiput(42.81,10.65)(0.49,-0.09){1}{\line(1,0){0.49}}
\multiput(43.3,10.57)(0.49,-0.08){1}{\line(1,0){0.49}}
\multiput(43.79,10.48)(0.49,-0.07){1}{\line(1,0){0.49}}
\multiput(44.28,10.41)(0.49,-0.07){1}{\line(1,0){0.49}}
\multiput(44.78,10.34)(0.49,-0.06){1}{\line(1,0){0.49}}
\multiput(45.27,10.28)(0.5,-0.06){1}{\line(1,0){0.5}}
\multiput(45.77,10.22)(0.5,-0.05){1}{\line(1,0){0.5}}
\multiput(46.27,10.17)(0.5,-0.04){1}{\line(1,0){0.5}}
\multiput(46.76,10.13)(0.5,-0.04){1}{\line(1,0){0.5}}
\multiput(47.26,10.09)(0.5,-0.03){1}{\line(1,0){0.5}}
\multiput(47.76,10.06)(0.5,-0.02){1}{\line(1,0){0.5}}
\multiput(48.26,10.04)(0.5,-0.02){1}{\line(1,0){0.5}}
\multiput(48.75,10.02)(0.5,-0.01){1}{\line(1,0){0.5}}
\multiput(49.25,10.01)(0.5,-0.01){1}{\line(1,0){0.5}}
\put(49.75,10){\line(1,0){0.5}}
\multiput(50.25,10)(0.5,0.01){1}{\line(1,0){0.5}}
\multiput(50.75,10.01)(0.5,0.01){1}{\line(1,0){0.5}}
\multiput(51.25,10.02)(0.5,0.02){1}{\line(1,0){0.5}}
\multiput(51.74,10.04)(0.5,0.02){1}{\line(1,0){0.5}}
\multiput(52.24,10.06)(0.5,0.03){1}{\line(1,0){0.5}}
\multiput(52.74,10.09)(0.5,0.04){1}{\line(1,0){0.5}}
\multiput(53.24,10.13)(0.5,0.04){1}{\line(1,0){0.5}}
\multiput(53.73,10.17)(0.5,0.05){1}{\line(1,0){0.5}}
\multiput(54.23,10.22)(0.5,0.06){1}{\line(1,0){0.5}}
\multiput(54.73,10.28)(0.49,0.06){1}{\line(1,0){0.49}}
\multiput(55.22,10.34)(0.49,0.07){1}{\line(1,0){0.49}}
\multiput(55.72,10.41)(0.49,0.07){1}{\line(1,0){0.49}}
\multiput(56.21,10.48)(0.49,0.08){1}{\line(1,0){0.49}}
\multiput(56.7,10.57)(0.49,0.09){1}{\line(1,0){0.49}}
\multiput(57.19,10.65)(0.49,0.09){1}{\line(1,0){0.49}}
\multiput(57.68,10.74)(0.49,0.1){1}{\line(1,0){0.49}}
\multiput(58.17,10.84)(0.49,0.1){1}{\line(1,0){0.49}}
\multiput(58.66,10.95)(0.49,0.11){1}{\line(1,0){0.49}}
\multiput(59.14,11.06)(0.48,0.12){1}{\line(1,0){0.48}}
\multiput(59.63,11.18)(0.48,0.12){1}{\line(1,0){0.48}}
\multiput(60.11,11.3)(0.48,0.13){1}{\line(1,0){0.48}}
\multiput(60.59,11.43)(0.48,0.14){1}{\line(1,0){0.48}}
\multiput(61.07,11.56)(0.48,0.14){1}{\line(1,0){0.48}}
\multiput(61.55,11.7)(0.48,0.15){1}{\line(1,0){0.48}}
\multiput(62.03,11.85)(0.47,0.15){1}{\line(1,0){0.47}}
\multiput(62.5,12)(0.47,0.16){1}{\line(1,0){0.47}}
\multiput(62.98,12.16)(0.47,0.16){1}{\line(1,0){0.47}}
\multiput(63.45,12.33)(0.47,0.17){1}{\line(1,0){0.47}}
\multiput(63.91,12.5)(0.47,0.18){1}{\line(1,0){0.47}}
\multiput(64.38,12.67)(0.23,0.09){2}{\line(1,0){0.23}}
\multiput(64.85,12.86)(0.23,0.09){2}{\line(1,0){0.23}}
\multiput(65.31,13.04)(0.23,0.1){2}{\line(1,0){0.23}}
\multiput(65.77,13.24)(0.23,0.1){2}{\line(1,0){0.23}}
\multiput(66.22,13.44)(0.23,0.1){2}{\line(1,0){0.23}}
\multiput(66.68,13.64)(0.23,0.11){2}{\line(1,0){0.23}}
\multiput(67.13,13.85)(0.22,0.11){2}{\line(1,0){0.22}}
\multiput(67.58,14.07)(0.22,0.11){2}{\line(1,0){0.22}}
\multiput(68.03,14.29)(0.22,0.11){2}{\line(1,0){0.22}}
\multiput(68.47,14.52)(0.22,0.12){2}{\line(1,0){0.22}}
\multiput(68.91,14.75)(0.22,0.12){2}{\line(1,0){0.22}}
\multiput(69.35,14.99)(0.22,0.12){2}{\line(1,0){0.22}}
\multiput(69.78,15.23)(0.22,0.12){2}{\line(1,0){0.22}}
\multiput(70.22,15.48)(0.21,0.13){2}{\line(1,0){0.21}}
\multiput(70.64,15.74)(0.21,0.13){2}{\line(1,0){0.21}}
\multiput(71.07,16)(0.21,0.13){2}{\line(1,0){0.21}}
\multiput(71.49,16.26)(0.21,0.14){2}{\line(1,0){0.21}}
\multiput(71.91,16.53)(0.21,0.14){2}{\line(1,0){0.21}}
\multiput(72.33,16.81)(0.21,0.14){2}{\line(1,0){0.21}}
\multiput(72.74,17.09)(0.2,0.14){2}{\line(1,0){0.2}}
\multiput(73.15,17.38)(0.2,0.15){2}{\line(1,0){0.2}}
\multiput(73.55,17.67)(0.2,0.15){2}{\line(1,0){0.2}}
\multiput(73.95,17.96)(0.13,0.1){3}{\line(1,0){0.13}}
\multiput(74.35,18.27)(0.13,0.1){3}{\line(1,0){0.13}}
\multiput(74.74,18.57)(0.13,0.1){3}{\line(1,0){0.13}}
\multiput(75.13,18.88)(0.13,0.11){3}{\line(1,0){0.13}}
\multiput(75.52,19.2)(0.13,0.11){3}{\line(1,0){0.13}}
\multiput(75.9,19.52)(0.13,0.11){3}{\line(1,0){0.13}}
\multiput(76.28,19.84)(0.12,0.11){3}{\line(1,0){0.12}}
\multiput(76.65,20.17)(0.12,0.11){3}{\line(1,0){0.12}}
\multiput(77.02,20.51)(0.12,0.11){3}{\line(1,0){0.12}}
\multiput(77.39,20.85)(0.12,0.11){3}{\line(1,0){0.12}}
\multiput(77.75,21.19)(0.12,0.12){3}{\line(1,0){0.12}}
\multiput(78.11,21.54)(0.12,0.12){3}{\line(1,0){0.12}}
\multiput(78.46,21.89)(0.12,0.12){3}{\line(0,1){0.12}}
\multiput(78.81,22.25)(0.11,0.12){3}{\line(0,1){0.12}}
\multiput(79.15,22.61)(0.11,0.12){3}{\line(0,1){0.12}}
\multiput(79.49,22.98)(0.11,0.12){3}{\line(0,1){0.12}}
\multiput(79.83,23.35)(0.11,0.12){3}{\line(0,1){0.12}}
\multiput(80.16,23.72)(0.11,0.13){3}{\line(0,1){0.13}}
\multiput(80.48,24.1)(0.11,0.13){3}{\line(0,1){0.13}}
\multiput(80.8,24.48)(0.11,0.13){3}{\line(0,1){0.13}}
\multiput(81.12,24.87)(0.1,0.13){3}{\line(0,1){0.13}}
\multiput(81.43,25.26)(0.1,0.13){3}{\line(0,1){0.13}}
\multiput(81.73,25.65)(0.1,0.13){3}{\line(0,1){0.13}}
\multiput(82.04,26.05)(0.15,0.2){2}{\line(0,1){0.2}}
\multiput(82.33,26.45)(0.15,0.2){2}{\line(0,1){0.2}}
\multiput(82.62,26.85)(0.14,0.2){2}{\line(0,1){0.2}}
\multiput(82.91,27.26)(0.14,0.21){2}{\line(0,1){0.21}}
\multiput(83.19,27.67)(0.14,0.21){2}{\line(0,1){0.21}}
\multiput(83.47,28.09)(0.14,0.21){2}{\line(0,1){0.21}}
\multiput(83.74,28.51)(0.13,0.21){2}{\line(0,1){0.21}}
\multiput(84,28.93)(0.13,0.21){2}{\line(0,1){0.21}}
\multiput(84.26,29.36)(0.13,0.21){2}{\line(0,1){0.21}}
\multiput(84.52,29.78)(0.12,0.22){2}{\line(0,1){0.22}}
\multiput(84.77,30.22)(0.12,0.22){2}{\line(0,1){0.22}}
\multiput(85.01,30.65)(0.12,0.22){2}{\line(0,1){0.22}}
\multiput(85.25,31.09)(0.12,0.22){2}{\line(0,1){0.22}}
\multiput(85.48,31.53)(0.11,0.22){2}{\line(0,1){0.22}}
\multiput(85.71,31.97)(0.11,0.22){2}{\line(0,1){0.22}}
\multiput(85.93,32.42)(0.11,0.22){2}{\line(0,1){0.22}}
\multiput(86.15,32.87)(0.11,0.23){2}{\line(0,1){0.23}}
\multiput(86.36,33.32)(0.1,0.23){2}{\line(0,1){0.23}}
\multiput(86.56,33.78)(0.1,0.23){2}{\line(0,1){0.23}}
\multiput(86.76,34.23)(0.1,0.23){2}{\line(0,1){0.23}}
\multiput(86.96,34.69)(0.09,0.23){2}{\line(0,1){0.23}}
\multiput(87.14,35.15)(0.09,0.23){2}{\line(0,1){0.23}}
\multiput(87.33,35.62)(0.18,0.47){1}{\line(0,1){0.47}}
\multiput(87.5,36.09)(0.17,0.47){1}{\line(0,1){0.47}}
\multiput(87.67,36.55)(0.16,0.47){1}{\line(0,1){0.47}}
\multiput(87.84,37.02)(0.16,0.47){1}{\line(0,1){0.47}}
\multiput(88,37.5)(0.15,0.47){1}{\line(0,1){0.47}}
\multiput(88.15,37.97)(0.15,0.48){1}{\line(0,1){0.48}}
\multiput(88.3,38.45)(0.14,0.48){1}{\line(0,1){0.48}}
\multiput(88.44,38.93)(0.14,0.48){1}{\line(0,1){0.48}}
\multiput(88.57,39.41)(0.13,0.48){1}{\line(0,1){0.48}}
\multiput(88.7,39.89)(0.12,0.48){1}{\line(0,1){0.48}}
\multiput(88.82,40.37)(0.12,0.48){1}{\line(0,1){0.48}}
\multiput(88.94,40.86)(0.11,0.49){1}{\line(0,1){0.49}}
\multiput(89.05,41.34)(0.1,0.49){1}{\line(0,1){0.49}}
\multiput(89.16,41.83)(0.1,0.49){1}{\line(0,1){0.49}}
\multiput(89.26,42.32)(0.09,0.49){1}{\line(0,1){0.49}}
\multiput(89.35,42.81)(0.09,0.49){1}{\line(0,1){0.49}}
\multiput(89.43,43.3)(0.08,0.49){1}{\line(0,1){0.49}}
\multiput(89.52,43.79)(0.07,0.49){1}{\line(0,1){0.49}}
\multiput(89.59,44.28)(0.07,0.49){1}{\line(0,1){0.49}}
\multiput(89.66,44.78)(0.06,0.49){1}{\line(0,1){0.49}}
\multiput(89.72,45.27)(0.06,0.5){1}{\line(0,1){0.5}}
\multiput(89.78,45.77)(0.05,0.5){1}{\line(0,1){0.5}}
\multiput(89.83,46.27)(0.04,0.5){1}{\line(0,1){0.5}}
\multiput(89.87,46.76)(0.04,0.5){1}{\line(0,1){0.5}}
\multiput(89.91,47.26)(0.03,0.5){1}{\line(0,1){0.5}}
\multiput(89.94,47.76)(0.02,0.5){1}{\line(0,1){0.5}}
\multiput(89.96,48.26)(0.02,0.5){1}{\line(0,1){0.5}}
\multiput(89.98,48.75)(0.01,0.5){1}{\line(0,1){0.5}}
\multiput(89.99,49.25)(0.01,0.5){1}{\line(0,1){0.5}}

\linethickness{0.3mm}
\multiput(50,90)(0.12,-0.12){333}{\line(1,0){0.12}}
\linethickness{0.3mm}

\linethickness{0.3mm}
\put(10,50){\line(1,0){80}}
\linethickness{0.3mm}
\multiput(50,10)(0.12,0.12){333}{\line(1,0){0.12}}
\linethickness{0.3mm}
\multiput(75,80)(0.12,-0.24){125}{\line(0,-1){0.24}}
\linethickness{0.3mm}
\multiput(20,75)(0.34,-0.12){208}{\line(1,0){0.34}}
\linethickness{0.3mm}
\multiput(20,25)(0.34,0.12){208}{\line(1,0){0.34}}
\linethickness{0.3mm}
\multiput(75,20)(0.12,0.24){125}{\line(0,1){0.24}}
\end{picture}
\end{center}

The following problem is well-known. Consider a regular polygon inscribed in the unit circle with one vertex at the point $(1,0)$.   Find  the product of the distances from  the point $(1,0)$ to the other $n-1$ vertices.  The answer is $n$.

\noindent \textit{Proof.}  Consider the unit circle $|z|=1$ in the complex plane.
Let $\omega=e^{ \frac{ 2 \pi i}{n} }$. Then the other vertices of the polygon are the roots of unity $\omega, \omega^2, \dots, \omega^{n-1}$.
These numbers together with $1$ are roots of the polynomial $z^n-1$.  Therefore
\bd z^n -1 =(z-1)(z-\omega) (z- \omega^2) \dots (z-\omega^{n-1})  \ . \ed
Divide both sides by $z-1$ to obtain
\bd z^{n-1}+z^{n-2}+ \dots +z+1 =(z-\omega) (z- \omega^2) \dots (z-\omega^{n-1}) \ . \ed
Substitute $z=1$ and take absolute values to obtain the result. \hspace{2 cm} $\blacksquare$

Let $\theta=\frac{2 \pi  }{n}$. Then $\omega^k=e^{i k \theta}$ and the distance from $1$ to $\omega^k$  is
\bd |1 -\omega^k|=|1- \left( \cos {k \theta}+i \sin{k   \theta} \right)|=\sqrt{ 2- 2 \cos {k  \theta}}=2 \sin \frac{ k \theta}{2} =2 \sin{\frac{\pi k}{n} } \ . \ed

As a result, we obtain  the formula

\be \label{sines} \prod_{k=1}^{n-1} \sin {    \frac{k \pi }{n}  } = \frac{n}{2^{n-1}}   \ . \ee

This formula was standard in textbooks in the 19th century. For example, it is derived in S. L. Loney's book Plane Trigonometry, one of the books that Ramanujan used to learn mathematics.

We are now in a position to give  a different proof of the $A_n$ sine  formula (\ref{det}).
Consider the formula (\ref{sines}) with  $n$ replaced by $2(n+1)$.    It becomes

\bd \prod_{k=1}^{2n+1} \sin {    \frac{k \pi }{2(n+1)} } = \frac{n+1}{2^{2n}}=\frac{n+1}{4^n}    \ . \ed

\noindent We split the product on the left-hand side in the following way

\bd \prod_{k=1}^{n} \sin {   \frac{k \pi }{2n+2}  } \prod_{k=n+2}^{2n+1} \sin {    \frac{k \pi }{2n+2}  } \ . \ed
Note also that
\bd \prod_{k=n+2}^{2n+1} \sin {    \frac{k \pi }{2n+2}  } =\prod_{k=n+2}^{2n+1} \sin   \left( \pi-  \frac{k \pi }{2n+2}  \right) = \prod_{k=n+2}^{2n+1} \sin {   \frac{2n+2-k  }{2n+2} \pi }   \ed
\bd =\prod_{j=n}^{2n-1} \sin {    \frac{ 2n-j  }{2n+2} \pi  } \ . \ed

\noindent This last product is the same as
\bd \prod_{k=1}^{n} \sin {    \frac{k \pi  }{2n+2}   }  \ , \ed
in reversed order.  Consequently,

\bd \prod_{k=1}^{n} \sin^2{    \frac{k \pi }{2(n+1)}  }=\frac{n+1}{4^n}  \  ,  \ed
which yields  exactly formula (\ref{det}).

In table \ref{detlist} we list the determinants for the Cartan matrices of simple Lie algebras.

\begin{table}[detlist]
\caption{Determinants for Cartan matrices}  \label{detlist}
\begin{center}
\begin{tabular}{|c|c|c|c|c|c|c|c|c|c|}
  \hline
  Lie algebra  & $A_n$ & $B_n$ & $C_n$ & $D_n$ & $E_6$ & $E_7$ & $E_8$& $F_4$ & $G_2$ \cr

  \hline

 {\rm Det}  & $n+1$& 2 &  2 & 4 & 3&2&1& 1&1  \\
    \hline
\end{tabular}
\end{center}

\end{table}

\subsection{Chebyshev polynomials of the second kind} \label{csk}

The Chebyshev polynomials form  an infinite sequence of orthogonal polynomials.  The  Chebyshev polynomial of the second kind of degree $n$ is usually denoted by $U_n$. We list some properties of Chebyshev polynomials following \cite{mason}, \cite{wiki}.

A fancy way to define the $n$th Chebyshev polynomial of the second kind is

\bd
U_n(x)={\rm det} \ \begin{pmatrix} 2x & 1&  & & &  \cr
                     1 &2x&1 & &&  \cr
                       & \ddots &\ddots \ddots & & & \cr
                       & & \ddots & \ddots & \ddots&  \cr
                        & & & 1&2x& 1  \cr
                        & & & & 1 &2 x \end{pmatrix} \ ,
                        \ed
where $n$ is the size of the matrix. By expanding the determinant with respect to the first row we get the  recurrence
\bd U_0(x)=1, \ \ U_1(x)=2x, \ \  U_{n+1}(x)=2x U_n(x)- U_{n-1}(x)   \ .\ed

It is easy then to compute recursively the first few polynomials:
\bd
 \begin{array}{lcl}
U_0(x) & =& 1 \\
U_1(x) & =& 2x \\
U_2(x) & = &4x^2-1 \\
U_3(x)&=&8x^3-4x \\
U_ 4(x) & =& 16 x^4-12x^2+1 \\
U_ 5(x) & = & 32x^5-32x^3+6x \\
U_ 6(x) & =& 64x^6-80x^4+24 x^2 -1  \ .
\end{array}
 \ed

Letting $x=\cos{\theta}$ we obtain
\bd U_n(x)=\frac { \sin((n+1)\theta) } {     \sin \theta} \ .\ed

There is also an explicit formula

\be \label{Ubinom}
U_n(x)=\sum_{j=0}^n (-2)^j  \binom{ n+j+1}{ 2j+1} (1-x)^j    \  \ .
\ee

Another formula in powers of $x$ is
\be \label{Uproduct}
U_n(x)=\sum_{j=0}^{ [\frac{n}{2}]} (-1)^j \binom{n-j}{j} (2x)^{ n-2j}  \ .
\ee

The polynomials $U_n$ satisfy the following properties:
\bd
 \begin{array}{lcl}
U_n(-x)&=&(-1)^n U_n(x)\\
U_n(1)&=&n+1 \\
U_{2n}(0)&=&(-1)^n\\
U_{2n-1}(0)&=&0  \  .
\end{array}
 \ed

Knowledge of the roots of $U_n$ implies
\be \label{rootsan}
U_n(x)=2^n \prod_{j=1}^n \left[ x-\cos \left( \frac{j\pi}{n+1}\right) \right]   \ .
\ee

Setting $x=1$ in this equation  we obtain again the $A_n$ sine  formula (\ref{det}).

\begin{lemma}
\label{cheban}
\bd p_n(x)=U_n (x/2-1)  \ , \ed
where $U_n$ is the Chebyshev polynomial of the second kind.

\end{lemma}

\begin{proof}
We write the eigenvalue equation in the form
det$(xI_n - C_{A_n})=0$ where $I_n$ is the $n \times n$ identity matrix.  Explicitly,

\bd
{\rm det} \ ( x I_n -C_{A_n})={\rm det} \ \begin{pmatrix} x-2 & 1&  & & &  \cr
                     1 &x-2 &1 & &&  \cr
                       & \ddots &\ddots \ddots & & & \cr
                       & & \ddots & \ddots & \ddots&  \cr
                        & & & 1&x-2& 1  \cr
                        & & & & 1 & x-2 \end{pmatrix} =    \ed
                        \bd {\rm det} \ \begin{pmatrix} 2 \left(\frac{x-2}{2}\right) & 1&  & & &  \cr
                     1 &2 \left(\frac{x-2}{2}\right) &1 & &&  \cr
                       & \ddots &\ddots \ddots & & & \cr
                       & & \ddots & \ddots & \ddots&  \cr
                        & & & 1&  2 \left(\frac{x-2}{2}\right)& 1  \cr
                        & & & & 1 & 2 \left(\frac{x-2}{2}\right) \end{pmatrix}=U_n\left( \frac{x}{2} -1 \right)   \ . \ed
\end{proof}

\begin{remark}
Note that
\bd  p_n(0)=U_n(-1)=(-1)^n U_n(1)=(-1)^n (n+1)  \ , \ed
which agrees (up to a sign)  with the formula for the determinant of $A_n$.
\end{remark}

We list the  formula for the characteristic polynomial of the matrix $A_n$ for small values of $n$.

\begin{eqnarray*}
\begin{array}{rcl}
p_1(x)&=& x-2 \\
p_2(x) & = & x^2-4x+3=(x-1)(x-3) \\
p_3(x)& = & x^3-6 x^2+10x -4=(x-2)(x^2-4x+2) \\
p_4(x) & = & x^4-8x^3+21 x^2-20x +5 =(x^2-5x+5)(x^2-3x+1)\\
p_5(x) &=&  x^5-10 x^4+36 x^3 -56 x^2 +35 x -6=(x-1)(x-2)(x-3)(x^2-4x+1) \\
p_6(x)&=& x^6-12 x^5+55 x^4-120 x^3 +126 x^2 -56x +7 \\
p_7(x)&=& x^7-14 x^6+78 x^5- 220 x^4 +330 x^3-252 x^2 +84x -8 \ .
\end{array}
\end{eqnarray*}

Using formula (\ref{Ubinom}) we  easily prove  the following explicit formula.

\begin{proposition} \label{charan}
Let $p_n(x)$ be the characteristic polynomial of the Cartan matrix (\ref{cartan}).  Then
\bd p_n(x)=\sum_{j=0}^n  (-1)^{n+j}  \binom{ n+j+1}{ 2j+1} x^j  \ . \ed

\end{proposition}

\section{Cartan matrix of type $B_n$ and $C_n$ } \label{cbn}

\subsection{Chebyshev polynomials of the first kind}

The Chebyshev polynomials of the first kind are denoted by $T_n(x)$.

They are defined by  the  recurrence
\bd T_0(x)=1, \ \ T_1(x)=x, \ \  T_{n+1}(x)=2x T_n(x)- T_{n-1}(x)   \ .\ed

Using this recursion one may compute  the first few polynomials:
\bd
 \begin{array}{lcl}
T_0(x) & =& 1 \\
T_1(x) & =& x \\
T_2(x) & = &2x^2-1 \\
T_3(x)&=&4x^3-3x \\
T_ 4(x) & =& 8 x^4-8x^2+1 \\
T_ 5(x) & = & 16x^5-20x^3+5x \\
T_ 6(x) & =& 32x^6-48x^4+18 x^2 -1  \ .
\end{array}
 \ed

It is well-known that $\cos(n\theta)$ can be expressed as a polynomial in $\cos(\theta)$.
For example

\bd
 \begin{array}{lcl}
 \cos(0\theta)&=& 1  \\
\cos(1\theta) & =& \cos \theta \\
\cos(2 \theta) & =& 2(\cos \theta)^2 -1  \\
\cos (3 \theta) & =& 4(\cos \theta)^3-3 (\cos \theta) \ .
\end{array}
\ed

More generally we have:
\bd \cos(n \theta)=T_n( \cos \theta)  \ .\ed

There is also an explicit formula

\be \label{Tbinom}
T_n(x)=n \sum_{j=0}^n (-2)^j \frac{ (n+j-1)!}{ (n-j)! (2j)!} (1-x)^j    \ \ (n>0)  \ .
\ee

Another formula in powers of $x$ is
\be \label{Tproduct}
T_n(x)= \frac{n}{2} \sum_{j=0}^{ [\frac{n}{2}]} (-1)^j \frac{ (n-j-1)!} { j! (n-2j)!}  (2x)^{ n-2j}  \ .
\ee

The polynomials $T_n$ satisfy the following properties:
\bd
 \begin{array}{lcl}
T_n(-x)&=&(-1)^n T_n(x)\\
T_n(1)&=&1 \\
T_{2n}(0)&=&(-1)^n\\
T_{2n-1}(0)&=&0  \  .
\end{array}
 \ed

Also
\be \label{rootsbn}
T_n(x)=2^{n-1} \prod_{j=1}^n \left[ x-\cos \left( \frac{(2j-1) \pi}{2n}\right) \right]   \ .
\ee

Setting $x=1$ in this equation  gives the formula
\bd \prod_{j=1}^n \sin^2 \frac{ (2j-1) \pi}{4n}= \frac{1}{ 2^{2n-1}}  \ .\ed

We can also write this formula in the form:

\be  \label{detbn}   2^{2n} \prod_{j=1}^n \sin^2 \frac{ (2j-1) \pi}{4n} = 2    \ .  \ \ \ \ \ \ \ \ \ \ \left( B_n \ \ \   {\rm sine \ \ \  formula} \right) \ee

We refer to this relation as the $B_n$ sine formula.

\subsection{The characteristic polynomial}
The Cartan matrix of type $B_n$ is a tri-diagonal matrix of the form
\be \label{cartanbn} C_{B_n}= \begin{pmatrix} 2 & -1&  & & &  \cr
                     -1 &2 &-1& &&  \cr
                       & \ddots &\ddots \ddots & & & \cr
                       & & -1& 2 & -1&  \cr
                        & & & -1&2& -2  \cr
                        & & & & -1 & 2\end{pmatrix} \ .  \ee

Since the Cartan matrix of type $C_n$ is the transpose of this matrix we consider only the Cartan matrix of type $B_n$. Using expansion on the first row it easy to prove that  ${\rm det}\ (C_{B_n})=2$.

We list the formula for the characteristic polynomial of the matrix $B_n$ for small values of $n$.

\begin{eqnarray*}
\begin{array}{rcl}
p_2(x) & = & x^2-4x+2 \\
p_3(x)& = & x^3-6 x^2+9x -2 =(x-2)(x^2-4x+1)\\
p_4(x) & = & x^4-8x^3+20x^2-16x +2 \\
p_5(x) &=&  x^5-10 x^4+35 x^3 -50 x^2 +25 x -2=(x-2)(x^4-8x^3+19x^2-12x +1) \\
p_6(x)&=& (x^2-4x +1) (x^4-8x^3+20x^2-16x+1) \\
p_7(x)&=& x^7-14 x^6+77 x^5- 210 x^4 +294 x^3-196 x^2 +49x -2   \ .
\end{array}
\end{eqnarray*}

By  expanding the determinant of the matrix $2xI+A$ with respect to the first row, we obtain  the  recurrence
\bd q_1(x)=2x , \ \ q_2(x)=4x^2-2, \ \  q_{n+1}(x)=2x q_n(x)- q_{n-1}(x)   \ .\ed
One  may define $q_0(x)=2$.  The recurrence implies that  $q_n(x)=2 T_{n}(x)$ where $T_n$ is the $n$th Chebyshev polynomial of the first kind.

Using formula (\ref{Tbinom}) it is easy to show the following result.

\begin{proposition} \label{charbn}
Let $p_n(x)$ be the characteristic polynomial of the Cartan matrix (\ref{cartanbn}).  Then
\bd p_n(x)= \sum_{j=0}^{n-1}  (-1)^{n+j} \frac{ 2n  (n+j-1)!}{ (n-j)! (2j)!}  x^j  \ . \ed
\end{proposition}

\section{Cartan matrix of type $D_n$ }\label{cdn}

\subsection{The characteristic polynomial}
The Cartan matrix of type $D_n $ is a    matrix of the form
\be \label{cartandn} C_{D_n}=   \begin{pmatrix} 2 & -1&  & & &  \cr
                     -1 &2 &-1& &&  \cr
                       & \ddots &\ddots \ddots & & & \cr
                       & & -1& 2 & -1& -1 \cr
                        & & & -1&2& 0 \cr
                        & & & -1& 0& 2\end{pmatrix} \ .  \ee
   Note that the matrix is no longer tri-diagonal. Using expansion on the first row and induction it easy to prove that ${\rm det}\ (C_{D_n})=4$.

We list some formulas for the characteristic polynomial of the matrix $C_{D_n}$ for small values of $n$.

\begin{eqnarray*}
\begin{array}{rcl}
p_2(x)&=& x^2-4x+2=(x-2)^2 \\
p_3(x)& = & x^3-6 x^2+10x-4 =(x-2)(x^2-4x+2)\\
p_4(x) & = & x^4-8x^3+21x^2-20x+4=(x-2)^2(x^2-4x+1)\\
p_5(x) &=&  x^5-10 x^4+36 x^3 -56x^2 +34x-4=(x-2)(x^4-8x^3+20x^2-16x +2)\\
p_6(x)&=& x^6-12 x^5+55x^4-120 x^3 +125 x^2 -52x+4=(x-2)^2(x^4-8x^3+19x^2-12x +1) \ .

\end{array}
\end{eqnarray*}

\begin{proposition} \label{chardn}
Let $p_n(x)$ be the characteristic polynomial of the Cartan matrix (\ref{cartandn}).  Then
\bd p_n(x)= (x-2)\sum_{j=0}^{n-1} (-1)^{n+j} \frac{ (2n-2)(n+j-2)!} {(n-j-1)! (2j)! } x^j \ . \ed

\end{proposition}

By expanding the determinant of the matrix $2xI+A$ with respect to the first row, we deduce  the recurrence
\bd    \  q_2(x)=4x^2, \qquad   q_3(x)=8x^3-4x,  \qquad  q_{n+1}(x)=2x q_n(x)-q_{n-1}   \ .\ed
We may define $q_1(x)=4x$. It is clear that $q_n(x)=4x T_{n-1}(x)$ where $T_n$ is  the $n$th
 Chebyshev polynomial of the first kind.

\section{The Coxeter Polynomial}

A Coxeter graph is a simple graph $\Gamma$ with $n$ vertices  and edge weights $m_{ij} \in \{ 3, 4, \dots, \infty\}$.  We define $m_{ii}=1$ and  $m_{ij}=2$ if node $i$ is not connected with node $j$.  By convention if $m_{ij}=3$ then the edge is often not labeled.  If $\Gamma$ is a Coxeter graph with $n$ vertices we define a bilinear form $B$ on ${\bf R}^n$ by choosing a basis $e_1, e_2, \dots, e_n$ and setting
\bd B(e_i, e_j)=-2 \cos { \frac{\pi}{m_{ij}} } \ .\ed
If $m_{ij}=\infty$ we define $B(e_i, e_j)=-2$.   We also define for $i=1,2, \dots, n$ the reflection
\bd \sigma_i (e_j)=e_j - B(e_i, e_j) e_i \ . \ed
Let $S=\{ \sigma_i \ | \ i=1, \dots, n \} $.  The Coxeter group $W(\Gamma)$ is the group generated by the reflections in  $S$.  $W$ has the presentation

\bd W=\langle \sigma_1, \sigma_2, \dots, \sigma_n \ | \ \sigma_i^2=1, (\sigma_i \sigma_j)^{m_{ij}} =1 \rangle \ .\ed
It is well-known that $W(\Gamma)$ is finite if and only if $B$ is positive definite.
 A \textit{Coxeter element} (or transformation)  is a product of the form
 \bd \sigma_{\alpha(1)} \sigma_{\alpha(2)} \dots \sigma_{\alpha(n)}   \ \ \ \alpha \in S_n  \ .\ed
 If the Coxeter graph is a tree then the Coxeter elements are in a single conjugacy class in $W$.  A \textit{Coxeter polynomial}  for the Coxeter system $(W,S)$ is the characteristic polynomial of the matrix representation of a Coxeter element.  For Coxeter systems whose graphs are trees the Coxeter polynomial is uniquely determined. This covers all the cases we investigate.
We define $\rho(W)$  to be the spectral radius of the associated Coxeter adjacency matrix.
A Coxeter system is called

\begin{enumerate}
\item
Spherical if $\rho(W)<2$
\item
Affine if $\rho(W)=2$
\item
Hyperbolic or higher-rank if $\rho(W)>2$.
\end{enumerate}
In this paper we consider only spherical Coxeter systems. In this case $B$ is positive definite and the Coxeter element has finite order.

\begin{example}

Consider a Coxeter system with graph $A_3$.

\begin{picture}(30,30)(-60,40)
$
\put(25,20){\line(1,0){50}}
\put(85,20){\line(1,0){50}}
\put(20,20){\circle{10}}
\put(80,20){\circle{10}}
\put(140,20){\circle{10}}
\put(0,40){}
$
\end{picture}

\vskip 2 cm

The bilinear form is defined by the Cartan matrix
\bd  \begin{pmatrix}
 2 &  -1 & 0 \cr
  -1&  2 & -1 \cr
  0& -1 & 2 \end{pmatrix} \ .
  \ed

The reflection $\sigma_1$  is determined  by the action
\bd \sigma_1 (e_1)=-e_1, \ \ \ \sigma_1 (e_2)=e_1+e_2, \ \ \ \sigma_1(e_3)=e_3  \ . \ed
It has the  matrix representation

\bd \sigma_1 = \begin{pmatrix}
 -1 &  1 & 0 \cr
  0&  1 & 0 \cr
  0& 0 & 1 \end{pmatrix} \ .
  \ed
Similarly

\bd \sigma_2=  \begin{pmatrix}
 1 &  0 & 0 \cr
  1&  -1 & 1 \cr
  0& 0 & 1 \end{pmatrix} \ , \ \ \ \ \
  \sigma_3=  \begin{pmatrix}
 1 &  0 & 0 \cr
  0&  1 & 0 \cr
  0& 1 & -1 \end{pmatrix} \ .
  \ed

A Coxeter element is defined by
\bd R=\sigma_1 \sigma_2 \sigma_3=  \begin{pmatrix}
 0 &  0 & -1 \cr
  1&  0& -1 \cr
  0& 1 & -1 \end{pmatrix} \ .
  \ed
\end{example}

The Coxeter polynomial is  the characteristic polynomial of $R$ which is  \bd x^3+x^2+x+1 \ . \ed
  Note that $R$ has order $4$, $R^4=I$. The order of the Coxeter element is called the Coxeter number.  Note that the roots of the Coxeter polynomial are $-1, -i, i$. We can write $i=i^{1}, -1=i^2, -i=i^3$. The integers $2,3,4$ are the degrees of the Chevalley invariants and their product, $24$,  is the order of the Coxeter  group.

\subsection{Exponents}

In general the order of the Coxeter  group is
\bd (m_1+1) (m_2+1) \dots (m_l+1)  \ed
where $m_i$ are the exponents of the eigenvalues of the Coxeter polynomial. The factors $m_i+1$ are the degrees of the Chevalley invariants. If $\zeta$ is a primitive  $h$ root of unity (where $h$ is the Coxeter number) then the roots  of the Coxeter polynomial are $\zeta^m$ where $m$ runs over the exponents of the corresponding root system \cite{coleman2},  \cite{coxeter2}.  This observation allows the calculation of the Coxeter polynomial for each root system.

In table \ref{coxetern} we list the Coxeter number and the exponents  for each root system.

\begin{table}[coxetern]
  \caption{Exponents and Coxeter number for root systems.\label{coxetern}}
  \begin{center}
  \begin{tabular}{|c|c|c|}
    \hline
      \ Root system\ \ & Exponents\ \ &  Coxeter number\\
    \hline
      ${\mathcal A}_n$ &$1,2,3, \dots, n$ &$n+1$\\
      ${\mathcal B}_n$ &$1,3,5, \dots, 2n-1$ &$2n$\\
      ${\mathcal C}_n$ &$1,3,5, \dots, 2n-1 $&$2n$\\
      ${\mathcal D}_n$ &$1,3,5,...., 2 n-3, n-1$ &$2n-2$\\
      ${\mathcal E}_6$ & $1,4,5,7,8,11  $&$12$\\
      ${\mathcal E}_7$ &$1,5,7,9,11,13,17$ &$18$\\
       ${\mathcal E}_8$ &$1,7,11,13,17,19,23,29$ &$30$\\
       ${\mathcal F}_4$ &$1,5,7,11$ &$12$\\
        ${\mathcal G}_2$ &$1,5$ &$6$\\

    \hline
  \end{tabular}
  \end{center}
\end{table}

Let us recall the definition of exponents for a simple complex Lie group
$G$, see \cite{coleman}, \cite{collin},  \cite{kostant}.    We
form the de Rham cohomology groups $H^i (G, \mathbb{C})$ and the
corresponding Poincar\'e polynomial of $G$ as follows,
\begin{displaymath}
p_G(t)=\sum_{i=1}^{\ell}  b_i t^i \ ,
\end{displaymath}
 where $b_i=  $dim  $H^i (G, \mathbb{C})$ are the Betti numbers of $G$. The De Rham groups encode topological information about $G$.
 Following work of Cartan, Ponrjagin and Brauer,   Hopf proved that the
cohomology algebra is  isomorphic to that of  a finite product of $\ell$
 spheres of odd dimension  where $\ell$ is the rank of $G$. This result  implies that
\bd
p_G(t)= \prod_{i=1}^{\ell} (1+t^{2 m_i +1} )\ .
\ed
The  positive integers $\{ m_1, m_2, .... , m_{\ell} \}$ are called the
{\it exponents} of $G$. They are also the exponents of the Lie algebra $\mathfrak{g}$ of $G$.

One can also extract the exponents from
the root space decomposition of $ \mathfrak{g}$ following   methods  which were developed by R. Bott, A. Shapiro, R. Steinberg,  A. J. Coleman and B. Kostant. We describe a method which was observed empirically by A. Shapiro and R. Steinberg and proved in a uniform way by B. Kostant.
 If $\alpha \in \Delta^{+}$, then we can express $\alpha$ uniquely as a sum of simple roots, $\alpha=\sum_{i=1}^{\ell} n_i \alpha_i $, where $n_i$ are non-negative integers. The height of $\alpha$ is defined to be  the number
\bd {\rm ht}\, (\alpha)=\sum_{i=1}^{\ell} n_i \ . \ed

The  procedure  is the following.
\begin{proposition}
 Let $b_j$ be  the number of $\alpha \in \Delta^{+}$ such that ${\rm ht}\, (\alpha)=j$. Then $b_{j}-b_{j+1}$ is the number of times $j$ appears as an exponent of $\g$. Furthermore,
\bd\ell =b_1 \ge b_2 \ge \cdots \ge b_{h-1}  \ ,  \ed
where $h$ is the Coxeter number.  Moreover, the partition of $r=|\Delta^{+}|$ as defined by the above sequence is conjugate to the partition
\bd h-1=m_{\ell} \ge m_{\ell-1} \ge \cdots \ge m_1=1  \ . \ed
The Coxeter number is determined by the relation $h \ell =2r$.
\end{proposition}

\begin{example} {\rm  Consider the   example of  $A_3$. Recall that $\Pi=\{\alpha_1, \alpha_2, \alpha_3 \}$ and the set of positive roots $\Delta^{+}$  is given by \bd \Delta^{+}= \{ \alpha_1, \alpha_2, \alpha_3, \alpha_1+\alpha_2, \alpha_2+\alpha_3, \alpha_1+\alpha_2+\alpha_3  \} \ . \ed
Therefore $r=6$ and the Coxeter number is determined by $h \ell =2r$, i.e., $3h=12$, or $h=4$.  The number of roots of height $1$ is $3$, therefore $b_1=3$. Similarly, $b_2=2$ and $b_3=1$.    Therefore, the exponents are $m_3=3$, $m_2=2$ and $m_1=1$.}
\end{example}

The exponents of a  complex simple Lie algebra   are given in  Table \ref{coxetern}.

Note, in the Table, the well-known  duality of  the set of exponents:
\be  \label{dual} m_i+m_{\ell+1-i} = h  \ee
where $h$ is the Coxeter number.
\subsection{Cyclotomic polynomials}

 A complex number $\omega$ is called a {\it primitive
$n$th root of unity} provided $\omega$  is an $n$th root of unity
and has order $n$. Such an $\omega$ generates the group of units, i.e.,
$\omega, \omega^2, \dots, \omega^n=1 $ coincides with the set
of all roots of unity.  An $n$th root of $1$  of the form $\omega^k$ is a primitive
root of unity iff $k$ is relatively prime to $n$. Therefore the
number of primitive roots of unity  is equal to $\phi(n)$ where $\phi$ is
Euler's totient function.

We define the $n$th cyclotomic polynomial by
\begin{displaymath}
\Phi_n(x)=(x-\omega_1)(x-\omega_2) \cdots (x-\omega_s)  \ ,
\end{displaymath}
where $\omega_1, \omega_2, \ldots \omega_s$ are all the distinct
primitive $n$th roots of unity. The degree of $\Phi_n$ is of course equal to $s=\phi(n)$.
 $\Phi_n(x)$ is a monic, irreducible polynomial with integer coefficients.   For example $\Phi_1(x) =x-1$, $\Phi_2(x) = x+1$, $\Phi_3(x) =x^2+x+1$
  and $\Phi_ 4(x) = x^2+1 $.

   A basic formula for  cyclotomic polynomials is

\begin{equation}
\label{cyclo} x^n -1= \prod_{d|n} \Phi_d(x) \ ,
\end{equation}
where $d$ ranges over all positive divisors of $n$. This gives a
recursive method of calculating  cyclotomic polynomials.  For
example, if $n=p$,  where $p$ is prime, then $x^p -1=\Phi_1(x) \Phi_p
(x)$ which implies that \bd \Phi_p(x)=x^{p-1}+x^{p-2}+\dots+x+1 \
.\ed

In table \ref{coxeterpoly} we list the Coxeter polynomial and its factorization as a product of cyclotomic polynomials. The cases $B_n$ and $D_n$ are  explained in remark \ref{remarkbn}.

\begin{table}[coxeterpoly]
  \caption{Coxeter polynomials for Spherical Graphs\label{coxeterpoly}}
  \begin{center}
  \begin{tabular}{|c|c|c|}
    \hline
      \ Dynkin Diagram\ \ &Coxeter polynomial \ \ &\ Cyclotomic Factors\ \  \\
    \hline
      ${\mathcal A}_n$ & $x^n+x^{n-1}+\dots x+1$ &$\prod_{\substack{  d|(n+1)     \\  d\not=1}} \Phi_d$      \\
      ${\mathcal B}_n, {\mathcal C}_n$ & $x^n+1$ &$ \prod_{d|N} \Phi_{2md} (x)(*) $  \\
      ${\mathcal D}_n$ & $x^n+x^{n-1}+x+1$ & $ \Phi_2 \prod_{d|N} \Phi_{2md} (x)(**)$  \\
      ${\mathcal E}_6$ &$x^6+x^5-x^3+x+1$ &$ \Phi_3 \Phi_{12}$       \\
      ${\mathcal E}_7$ &$x^7+x^6-x^4-x^3+x+1$ &$\Phi_2 \Phi_{18}$  \\
      ${\mathcal E}_8$ &  $ x^8+x^7-x^5-x^4-x^3+x+1$  & $\Phi_{30}$ \\
       ${\mathcal F}_4$ &  $ x^4-x^2+1$  & $\Phi_{12}$ \\
        ${\mathcal G}_2$ &  $ x^2-x+1$  & $\Phi_6 $  \\
    \hline
  \end{tabular}
  \end{center}
\end{table}

\begin{remark} \label{remarkbn}
$(*)$ $x^n+1$  is irreducible iff $n$ is a power of $2$, see \cite{damianou2}. If $n$ is odd it has a factor of $x+1=\Phi_2(x)$.  To obtain the factorization of $x^{n}+1$ we use the following procedure. Let $n=2^{\alpha} p_1^{k_1} p_2^{k_2} \dots p_s^{k_s}$ where the  $p_i$ are odd primes.   Define $m=2^{\alpha}$ and $N=p_1^{k_1} \dots p_s^{k_s}$.  The factorization of $x^n+1$ is given by
\bd x^n+1= \prod_{d|N} \Phi_{2md} (x)  \ . \ed

\noindent
$(**)$ In the case of $D_n$ we use the same procedure with $n$ replaced by $n-1$.
\end{remark}

\begin{remark}
In table \ref{coxeterpoly} note the following: In the  factorization of $f_n(x)$, the factor of highest degree is $\Phi_h$ where $h$ is the Coxeter number.

\end{remark}

The roots of $U_n(x)$  are given by

\bd x_k= \cos \left( \frac{k \pi}{n+1} \right)  \ \ \ \ \ k=1,2, \dots, n   \ .  \ed

Therefore the roots of \bd a_n(x)=U_n \left( \frac{x}{2} \right)  \ed
are in the interval $[-2, 2]$. If a  monic polynomial with integer coefficients  has all of  its roots in the interval $[-2,2]$ then
they  are of a special form. We follow the proof from \cite{jones}.

\begin{proposition}\label{kron}
Let $\lambda $ be a non-zero real  root of a monic polynomial $p(x) \in {\bf Z}[x]$. If all the roots of $p(x)$ are real and in the interval $[-2,2]$ then $\lambda=2 \cos (2 \pi q)$  where $q$ is a rational number.
\end{proposition}
\begin{proof}
Let $n=deg(p)$ and define the associated polynomial
\bd Q(x)=x^n p(x+ \frac{1}{x})  \ .  \ed
Denote by  \bd \lambda=2 \cos \theta_1,  2 \cos \theta_2,  \dots 2 \cos \theta_n    \ed
the roots of $p(x)$.  Then
\bd p(x)=\prod_j \left( x-2 \cos \theta_j \right)  \ed
\bd Q(x) = \prod_j (x^2- 2x \cos \theta_j +1) =\prod_j (x-e^{i \theta_j} ) (x- e^{ -i\theta_j})  \ed
It follows from Kronecker's Theorem, see \cite{damianou},  that $e^{i\theta_1}$ is a root of unity and therefore $\frac{\theta_1}{2 \pi}$  is rational.
\end{proof}

Given a graph $\Gamma$ we  define the associated polynomial $Q_n(x)$ by the formula
\bd Q_n(x)=x^n a_n (x + \frac{1}{x})  \ ,\ed
where $a_n(x)$ is the characteristic polynomial of the graph.  It is clear that $Q_n(x)$ is a reciprocal polynomial. In all the cases we consider $Q_n(x)$ is an even  polynomial.

\subsection{Associated polynomials for $A_n$ }

We present the factorization of $Q_n(x)$ for small values of $n$ as a product of cyclotomic polynomials. Note that $Q_n(x)$ is an even polynomial. The reason: If $n$ is even $U_n(x)$ is an even polynomial and  $a_n(x)=U_n( \frac{x}{2})$ is also even. Therefore $Q_n(x)$ is  even. If $n$ is odd then $U_n(x)$ and $a_n(x)$ are both odd functions.  This implies that $Q_n(x)$ is even.

\begin{itemize}
\item  $A_2$ \ \    $\Phi_3 \Phi_6$
\item $A_3$  \ \   $\Phi_4 \Phi_8$
\item $A_4$ \ \   $\Phi_5 \Phi_{10}$
\item  $A_5$ \ \  $\Phi_3 \Phi_4 \Phi_6 \Phi_{12} $
\item  $A_6$ \ \   $\Phi_7 \Phi_{14} $
\item  $A_7$ \ \  $\Phi_4 \Phi_8 \Phi_{16}$
\item  $A_8$ \ \  $\Phi_3 \Phi_6 \Phi_9 \Phi_{18}$
\item  $A_9$ \ \  $\Phi_4 \Phi_5 \Phi_{10} \Phi_{20}$
\item  $A_{10}$ \ \ $\Phi_{11} \Phi_{22} $
\item  $A_{11}$  \ \ $\Phi_3 \Phi_4 \Phi_6 \Phi_8 \Phi_{12} \Phi_{24}$
\end{itemize}

The characteristic polynomial of the Coxeter transformation has roots
$ \zeta^k$ where $\zeta$ is a primitive $h$ root of unity and $k$ runs over the exponents of a root system of type $A_n$.  Therefore
\bd f_n(x)=(x-\zeta)(x-\zeta^2) \dots (x-\zeta^n)  \ . \ed
\bd \Rightarrow (x-1)  f_n(x)=x^{n+1}-1  \ed
\bd \Rightarrow  f_n(x)=\frac{x^{n+1}-1 }{ x-1}  \ . \ed
Using formula (\ref{cyclo})  we obtain
\bd f_n(x)=\prod_{\substack{  d|{n+1}      \\  d\not=1}} \Phi_d  \ . \ed

It is not difficult to guess the factorization of $Q_n(x)$.

\begin{proposition}
\be  \label{prod}  Q_n(x)=\prod_{\substack{ j|2n+2 \\  j\not=1,2}} \Phi_j(x) \ . \ee
\end{proposition}

\begin{proof}
Since
\bd Q_n(x)=f_n(x^2)=\prod_{\substack{  j|{n+1}      \\  j\not=1}} \Phi_j(x)  \ed
we should know what is $\Phi_j(x^2)$.

It is well-known, see \cite{damianou2},  that
\[
\Phi_j(x^2)=\left\{\begin{array}{cl}
	\Phi_{2j}(x), & {\rm if} \ \ j \ \ {\rm is \ \  even} \\
	\Phi_j(x) \Phi_{2j} (x), & {\rm if} \ \ j \ \ {\rm is \ \  odd}
	   \end{array}\right.
\]

To complete the proof we must show that each divisor of $2n+2$ bigger than $2$  appears in the product (\ref{prod}).
Let $d$ be a divisor of $2n+2$ bigger than $2$.   We consider two cases:

\noindent i)  If $d$ is odd then since $d|2(n+1)$ we have that $d|n+1$.  Since $\Phi_d$ is a factor of $f_n(x)$,  then $f_d(x^2)=\Phi_d(x) \Phi_{2d} (x)$, and therefore $\Phi_d$ appears.

\bigskip
\noindent ii) If $d$ is even, then $d=2s$ for some integer $s$ bigger than $1$.  Since $2s|2(n+1)$ we have that $s|n+1$.  Therefore $\Phi_s$ appears in the factorization of $f_n(x)$.  If $s$ is odd then $\Phi_s(x^2)=\Phi_s(x) \Phi_{2s}(x)$ and if $s$ is even $\Phi_s(x^2)=\Phi_{2s}(x)$. In either case $\Phi_{2s}=\Phi_d$ appears.

\end{proof}

An alternative way to derive the  formula for $f_{A_n}$ is the following: Note that
the Coxeter adjacency matrix $A$ is related to the Cartan matrix with
$A=2I-C$ and $a_n(x)=p_n(x+2)=q_n(\frac{x+2}{2} -1)=q_n ( \frac{x}{2} )$.
In the case of $A_n$ we have
\bd a_n(x)=U_n( \frac{x}{2} ) \ . \ed
Therefore,
\bd Q_n(x)=x^n  U_n \left( \frac{1}{2} ( x+ \frac{1}{x}) \right)  \ . \ed

Set $x=e^{i\theta}$ to obtain

\begin{align*}
 Q_n(x)&= e^{in \theta} U_n \left( \frac{1}{2} ( e^{i\theta}+e^{-i \theta}) \right) \\
       & =e^{in \theta} U_n (\cos \theta) \\
       & =  e^{in \theta} \frac{\sin(n+1)\theta}{\sin \theta}  \\
       &=e^{in \theta} \frac{ \left( e^{i(n+1)\theta} +e^{-i(n+1) \theta }\right)}{e^{i\theta}-e^{-i \theta}} \\ &=\frac{x^{2(n+1)}-1}{x^2-1} \ . \end{align*}

       Setting   $u=x^2$ we obtain
\bd \frac{u^{n+1}- 1}{u-1}=u^n+u^{n-1}+ \dots + u+1  \ .\ed
This is of course the Coxeter polynomial of the $A_n$ graph.
Therefore   $Q_n(x)=x^{2n}+x^{2(n-1)}+ \dots + x^2+1$ for all $x \in {\bf C}$.

We present the characteristic polynomial of the adjacency matrix and the Coxeter polynomial for small values of $n$.
\begin{itemize}

\item $A_2$ \ \   $a_2:=x^2-1$ \ \ \ \ \ $f_2=\Phi_3$
\item $A_3$  \ \   $a_3:=x^3-2x $ \ \ \ \ \ $f_3=\Phi_2 \Phi_4$
\item  $A_4$ \ \  $a_4=x^4-3x^2+1  $ \ \ \ \ \ $f_4= \Phi_5$
\item  $A_5$ \ \  $a_5=x^5-4x^3+3x$ \ \ \ \ \  $f_5= \Phi_2 \Phi_3 \Phi_6 $
\item  $A_6$ \ \  $a_6=x^6-5x^4+6x^2-1  $ \ \ \ \ \ $f_6=\Phi_7$
\item  $A_7$ \ \  $a_7=x^7-6 x^5+10x^3-4x $ \  \ \ \ \ $f_7= \Phi_2 \Phi_4 \Phi_8$
\item  $A_8$ \ \  $a_8=x^8-7x^6+15x^4-10x^2+1$ \ \ \ \ \ $f_8=\Phi_3 \Phi_9$
\item  $A_{9}$ \ \ $a_9=x^9-8x^7+21x^5 -20 x^3+5x $ \ \ \ \ \ $f_{9}=\Phi_2 \Phi_5 \Phi_{10} $
\item  $A_{10}$ \ \ $a_{10}=x^{10}-9x^8+26x^6-35x^4+15x^2-1 $ \ \ \ \ \ $f_{10}=\Phi_{11}  \ .$
\end{itemize}

Note  that $a_n(x)$ is explicitly  given by the formula

\bd a_n(x)=U_n (\frac{x}{2})=\sum_{j=0}^{ [\frac{n}{2}]} (-1)^j \binom{n-j}{j} (x)^{ n-2j}  \ed

due to (\ref{Uproduct}).

\subsection{Associated Polynomials for $B_n$ and $C_n$ }

\bigskip

In the case of $B_n$ we have
\bd a_n (x)=2 T_n ( \frac{x}{2} ) \ . \ed
Therefore,
\bd Q_n(x)=2x^n  T_n \left( \frac{1}{2} ( x+ \frac{1}{x}) \right)  \ . \ed
Set $x=e^{i\theta}$ to obtain

\begin{align*}Q_n(x)&=2 e^{in \theta} T_n \left( \frac{1}{2} ( e^{i\theta}+e^{-i \theta}) \right)  \\
&= 2e^{in \theta} T_n (\cos \theta) \\
&= 2 e^{in \theta} \cos{n \theta}  \\
&= 2 e^{in \theta} \frac{1}{2} \left( e^{in\theta} +e^{-in \theta} \right) =e^{2 in \theta}+1=x^{2n}+1 \ . \end{align*}

Therefore   $Q_n(x)=x^{2n}+1$ for all $x \in {\bf C}$.  As a result the Coxeter polynomial is $f_n(x)=x^n+1$.  We present the factorization of $f_n(x)$ for small values of $n$.
\begin{itemize}
\item  $B_2$ \ \  $a_2=x^2-2$\  \ \  \ \  $f_2=\Phi_4$
\item $B_3$  \ \  $a_3=x^3-3x$\ \ \ \ \  $f_2=\Phi_2 \Phi_{6}$
\item $B_4$ \ \   $a_4=x^4-4x^2 +2$ \ \ \ \ \ $f_4=\Phi_{8}$
\item  $B_5$ \ \  $a_5=x^5-5x^3+5x $ \ \ \ \ \ $f_5=\Phi_2 \Phi_{10}$
\item  $B_6$ \ \  $a_6=x^6-6 x^4+9x^2-2$ \ \ \ \ \  $f_6=\Phi_4 \Phi_{12} $
\item  $B_7$ \ \  $a_7=x^7-7x^5+14x^3-7x$ \ \ \ \ \ $f_7=\Phi_2  \Phi_{14}$
\item  $B_8$ \ \  $a_8=x^8-8x^6+20x^4-16x^2+2 $ \  \ \ \ \ $f_8= \Phi_{16}$
\item  $B_9$ \ \  $a_9=x^9-9 x^7+27x^5-30x^3+9x $ \ \ \ \ \ $f_9=\Phi_2 \Phi_{6} \Phi_{18}$
\item  $B_{10}$ \ \ $a_{10}=x^{10}-10x^8+35 x^6 -50x^4 +25x^2-2 $ \ \ \ \ \ $f_{10}=\Phi_{4} \Phi_{20} \ . $
\end{itemize}

Write $n=2^{\alpha} N$ where $N$ is odd.  As we already mentioned
\bd f_n(x)=x^n+1= \prod_{d|N} \Phi_{2md} (x)   \ , \ed
where $m=2^{\alpha}$.  Therefore
\bd f_n(x)=x^n+1=\prod_{\substack{ d|n \\  d \ {\rm odd}}} \Phi_{2^{\alpha+1}  d} (x)=\prod_{\substack{ d|N \\ }} \Phi_{2^{\alpha+1}  d} (x) \ . \ed

\begin{proposition} Let $r=2^{\alpha+2}$. Then
\be  \label{prodbn}  Q_n(x)=\prod_{\substack{ d|n \\  d \ {\rm odd}}} \Phi_{rd} (x) \ . \ee
\end{proposition}

\begin{proof}
It follows from the formula
\bd \Phi_k(x^2)=\Phi_{2k} (x)   \ed
when  $k$ is even.

\end{proof}

Note  that the  $a_n(x)$ polynomial  is explicitly  given in this case  by the formula

\bd a_n(x)= 2 T_n (\frac{x}{2})= \sum_{j=0}^{ [\frac{n}{2}]} (-1)^j \frac{ n (n-j-1)!} { j! (n-2j)!}  (x)^{ n-2j}   \ed

due to (\ref{Tproduct}).

Since $a_n(x)= 2 T_n ( \frac{x}{2})$  these polynomials satisfy the recursion
\bd a_{n+1}=x a_n(x)-a_{n-1} (x)  \ed
with $a_0(x)=2$ and $a_1(x)=x$.     We would like to mention a useful application  of these polynomials. One  can use them to express $x^n+x^{-n}$ as a function of $\zeta=x+\frac{1}{x}$.  For $x=e^{i\theta}$ it is just the expression of $2 \cos n \theta$ as a polynomial in $2 \cos \theta$. This polynomial is clearly $a_n(x)$, the adjacency polynomial  of $B_n$.

\begin{example} \label{zeta}
\bd \left(x+\frac{1}{x} \right)^2 =x^2 +\frac{1}{x^2}+2   \ . \ed
Therefore
\bd x^2 + \frac{1}{x^2}=\zeta^2 -2 =a_2(\zeta) \ . \ed
Similarly
\bd x^3 + \frac{1}{x^3}=\zeta^3-3 \zeta =a_3(\zeta) \ , \ed
\bd x^4 + \frac{1}{x^4}=\zeta^4-4 \zeta^2 +2=a_4(\zeta) \ . \ed

\end{example}

\subsection{Associated Polynomials for $D_n$  }

In the case of $D_n$ we have
\bd q_n(x)=4x T_{n-1}(x)  \ . \ed
 Therefore,
\bd a_n(x)=2x T_{n-1} ( \frac{x}{2} ) \ ,  \ed
and
\bd Q_n(x)=2x^n (x+ \frac{1}{x})  T_{n-1} \left( \frac{1}{2} ( x+ \frac{1}{x}) \right)  \ . \ed

Using the methods of the previous section we obtain    $Q_n(x)=x^{2n}+x^{2(n-1)}+x^2+1$.  We conclude that $f_n(x)=x^n+x^{n-1}+x+1$.  We present the formula for $a_n(x)$ and  the factorization of $f_n(x)$ for small values of $n$.
\begin{itemize}

\item $D_4$ \ \   $a_4=x^4-3x^2$ \ \ \ \ \ \ \ \ \ \ \ \  $f_4(x)=\Phi_2^2 \Phi_{6}$
\item  $D_5$ \ \  $a_5=x^5-4x^3+2x  $ \ \ \ \ \ $f_5= \Phi_2 \Phi_{8}$
\item  $D_6$ \ \  $a_6=x^6-5 x^4+5 x^2$ \ \ \ \ \  $f_6=\Phi_2^2 \Phi_{10} $
\item  $D_7$ \ \  $a_7=x^7-6x^5+9x^3-2x $ \ \ \ \ \ $f_7=\Phi_2  \Phi_4 \Phi_{12}$
\item  $D_8$ \ \  $a_8=x^8-7x^6+14 x^4 -7x^2 $ \  \ \ \ \ \ \ \ \ \ \ $f_8= \Phi_2^2 \Phi_{14} $
\item  $D_9$ \ \  $a_9=x^9-8x^7+20x^5-16x^3+2x $ \ \ \ \ \ $f_9=\Phi_2 \Phi_{16}$
\item  $D_{10}$ \ \ $a_{10}=x^{10}-9x^8+27x^6-30x^4+9x^2 $ \ \ \ \ \ $f_{10}=\Phi_2^2 \Phi_{6} \Phi_{18} \ . $
\end{itemize}

Write $n-1=2^{\alpha} N$ where $N$ is odd and $r=2^{\alpha+1}$.
\bd f_n(x)=(x+1)(x^{n-1}+1)= \Phi_2(x) \prod_{d|N} \Phi_{rd} (x)   \ . \ed

\begin{proposition}
\be  \label{proddn}  Q_n(x)=\Phi_4(x) \prod_{\substack{ d|{n-1} \\  d \ {\rm odd}}} \Phi_{2rd} (x) \ . \ee
\end{proposition}

\section{Factorization of Chebyshev polynomials}
\subsection{Chebyshev polynomials of second kind}

It is well-known that the  roots of $U_n$ are given by
\bd x_k=\cos {\frac{k \pi}{n+1} }  \ \ \ \ \  \ k=1,2, \dots, n  \ , \ed
as we already observed in (\ref{rootsan}).

We can write them in the form $x_k=\cos k \theta$,
where $\theta=\frac{\pi}{n+1}$.

The roots of $a_n(x)=U_n( \frac{x}{2})$ are
\bd \lambda_k=2 \cos {  \frac{k\pi}{n+1}     }=2 \cos {k \theta}   \ \ \ \ \  \ k=1,2, \dots, n  \ , \ed
i.e.
the roots of $a_n(x)$ are
\bd 2 \cos \frac{ m_i \pi}{h}   \ed
where $m_i$ are the exponents of $A_n$ and $h$ is the Coxeter number for $A_n$.

\begin{itemize}
\item $a_n(x)$

The roots of $a_n(x)=U_n( \frac{x}{2})$ are
\bd \lambda_k=2 \cos {  \frac{k\pi}{h}     }=2 \cos {k \theta}   \ \ \ \ \  \ k=1,2, \dots, n  \ . \ed

Denote them by
\bd \lambda_1=2\cos \theta_1=2\cos \theta, \ \ \  \lambda_2=2\cos \theta_2, \dots, \ \ \lambda_n=2\cos \theta_n   \ .\ed
Note that $\theta_k=k \theta$ and
\bd \theta_j + \theta_{n+1-j}=\pi  \ . \ed
This implies that $\lambda_j=-\lambda_{n-j+1}$.  As a result
\be \label{minusa}  \{ \lambda_j | j=1,2, \dots, n \}=\{-\lambda_j| j=1,2, \dots, n \}  \ . \ee

\item
$p_n(x)$

The roots of $p_n(x)$ are then
\bd \xi_k=2 + \lambda_k=2+ 2 \cos { \frac{k \pi }{n+1}}=2 +2 \cos \theta_k=4 \cos^2 \frac{\theta_k}{2}  \ . \ed
It follows from (\ref{minusa}) that the eigenvalues of $C$ occur in pairs $\{ 2+ \lambda, 2-\lambda \}$. This is a general result which holds for each Cartan matrix corresponding to a simple Lie algebra over ${\bf C}$, see \cite[p. 345]{moody} for a general proof.

\item
$Q_n(x)$

It follows from Theorem  \ref{kron} that the roots of $Q_n(x)$ are
\bd e^{i\theta_1}, e^{i\theta_2}, \dots, e^{i\theta_n}, e^{-i \theta_1}, e^{-i\theta_2}, \dots, e^{-i\theta_n} \ , \ed
or equivalently,
\bd e^{i\theta}, e^{2i\theta}, \dots, e^{ni\theta}, e^{(n+2)i \theta}, e^{(n+3)i \theta}, \dots, e^{(2n+1)i \theta} \ . \ed
Note that \bd e^{i \theta}=e^{ \frac{i \pi}{h}} =e^{  \frac{2 \pi i}{2h} }  \ . \ed
As a result  $e^{i \theta}$ is a $(2h)$th primitive root of unity and therefore  a root of the cyclotomic polynomial $\Phi_{2 h}$.  The other roots of this cyclotomic polynomial are of course $e^{ki \theta}$  where $(k, 2h)=1$. This determines $\Phi_{2h}$ as an irreducible factor of $Q_n(x)$. To determine the other irreducible factors we proceed as follows: The root
\bd e^{2 i \theta}=e^{\frac{2 \pi i }{h}   }  \ed
is a primitive $h$th root of unity and the other roots of $\Phi_h$ are $e^{ki \theta}$ where $(k, 2h)=2$.
In general for each $d$ which is a divisor of $2h$ (but $d \not=1,2$)  we form $\Phi_d$ by grouping together all the $e^{ki\theta}$ such that $(k, 2h)=\frac{2h}{d}$.  Since $e^{ (n+1)i \theta}$ and $e^{2ni \theta}$ do not appear as roots of $Q_n(x)$  the cyclotomic polynomials $\Phi_1$ and $\Phi_2$ do not appear in the factorization. You can consider this argument as another proof of  Proposition \ref{prod}.

\item
$a_n(x)$

The roots of $a_n(x)$ are of the form $e^{i\theta_k}+ e^{-i \theta_k}$, where $e^{i \theta_k}$ is a root of $Q_n(x)$. It is easy to see that each irreducible factor of $Q_n(x)$ determines an irreducible factor of $a_n(x)$ and conversely.  In fact this is the argument of Lehmer in \cite{lehmer}. If $\Phi_d$ is a cyclotomic factor of $Q_n(x)$ then $\Phi_d$ being a reciprocal polynomial it can be written in the form
\bd \Phi_d(x)=x^m \psi_d(x+ \frac{1}{x} )  \ed
 where $m={\rm deg} \psi_n=\frac{1}{2} \phi(d)$.  According to Lehmer the  polynomial  $\psi_d$ is irreducible.
 These polynomials are all the irreducible factors of $a_n(x)$. This  fact is easily established by looking at the roots of $a_n(x)$.
The way to determine the irreducible factors of $a_n(x)$ is the following: Start with a cyclotomic factor of $Q_n(x)$. For example, consider $\Phi_{2h}(x)$.  The roots of this polynomial are   $e^{ki \theta}$  where $(k, 2h)=1$. Take only $e^{ki\theta}$ with $(k, 2h)=1$ such that $1 \le k \le n$.  The corresponding roots of $a_n(x)$ are of course $2 \cos k \theta$,  $1 \le k \le n$ and $(k, 2h)=1$. This determines the polynomial $\psi_{2h}$.  Then we repeat this procedure with the other cyclotomic factors.
\end{itemize}
\begin{example}
To determine the factorization of $U_5(x)=32x^5-32x^3+6x$.   Since $n=5$, $h=n+1=6$ and $\theta=\frac{\pi}{6}$.
The roots of $U_5$ are

\bd  \cos \frac{\pi}{6} \ \ \  \cos\frac{2 \pi}{6} \ \ \  \cos \frac{3\pi}{6} \ \ \  \cos \frac{ 4 \pi}{6} \ \ \  \cos \frac{5 \pi}{6}  \ .  \ed

The roots of $a_5(x)$ are

\bd  \lambda_1=2 \cos \frac{\pi}{6} \ \ \  \lambda_2=2\cos\frac{2 \pi}{6} \ \ \ \lambda_3=2 \cos \frac{3\pi}{6} \ \ \  \lambda_4=2\cos \frac{ 4 \pi}{6} \ \ \  \lambda_5=2\cos \frac{5 \pi}{6}  \ .  \ed

The roots of $Q_5(x)$ are
\bd e^{ \frac{\pi }{6} }, e^{  \frac{2 \pi }{6} }, e^{  \frac{3 \pi }{6} }, e^{  \frac{4 \pi }{6} }, e^{  \frac{ 5\pi }{6} }, e^{ \frac{7\pi }{6} }, e^{ \frac{8 \pi }{6} }, e^{ \frac{ 9 \pi }{6} }, e^{ \frac{10 \pi }{6} }, e^{11 \frac{\pi }{6} } \ . \ed

We group together all $e^{ik \theta}$ such that $(k, 12)=1$ i.e.,
\bd e^{ \frac{\pi }{6} }, e^{ \frac{5 \pi }{6} }, e^{ \frac{7 \pi }{6} }, e^{ \frac{11 \pi }{6} } \  \ed
which are the roots of $\Phi_{12}(x)$.  Note that
\bd e^{ \frac{\pi }{6} }+e^{11 \frac{\pi }{6} }=e^{ \frac{\pi }{6} }+e^{- \frac{\pi }{6} }=2 \cos \frac{\pi}{6}=\lambda_1=\sqrt{3}   \ , \ed
and
\bd e^{ \frac{5 \pi }{6} }+e^{ \frac{7\pi }{6} }=e^{ \frac{5 \pi }{6} }+e^{ \frac{-5 \pi }{6} }=\cos \frac{5 \pi}{6}=\lambda_5 =-\sqrt{3} \ . \ed
These roots $\lambda_1$ and $\lambda_5$ are roots of $\psi_{12}=x^2-3$ which is an irreducible factor of $a_5(x)$.

Then we group together all $e^{ik \theta}$ such that $(k, 12)=2$ i.e.,
\bd e^{ \frac{2 \pi }{6} }, e^{ \frac{10 \pi }{6} } \ed
which are the roots of $\Phi_6(x)$.  Note that
\bd e^{ \frac{2 \pi }{6} }+ e^{ \frac{10 \pi }{6} }=e^{ \frac{2 \pi }{6} }+ e^{- \frac{2 \pi }{6} }=2\cos\frac{2 \pi}{6}=\lambda_2=1  \ . \ed
Therefore $x-1$ is the irreducible factor of $a_5(x)$ corresponding to  $\Phi_6(x)$.

Then we group together all $e^{ik \theta}$ such that $(k, 12)=3$ i.e.,
\bd e^{ \frac{3 \pi }{6} }, e^{ \frac{9 \pi }{6} } \ed
which are the roots of $\Phi_4(x)$.  Note that
\bd e^{ \frac{3 \pi }{6} }+ e^{ \frac{9 \pi }{6} }=e^{ \frac{3 \pi }{6} }+ e^{- \frac{3\pi }{6} }=2\cos\frac{3 \pi}{6}=\lambda_3 =0\ . \ed
Therefore $\psi_4(x)=x$ is the irreducible factor of $a_5(x)$ corresponding to  $\Phi_4(x)$.

Finally we group together all $e^{ik \theta}$ such that $(k, 12)=4$ i.e.,
\bd e^{ \frac{4 \pi }{6} }, e^{ \frac{8 \pi }{6} } \ed
which are the roots of $\Phi_3(x)$.  Note that
\bd e^{ \frac{4 \pi }{6} }+ e^{ \frac{8 \pi }{6} }=e^{ \frac{4 \pi }{6} }+ e^{- \frac{4 \pi }{6} }=2 \cos\frac{4 \pi}{6}=\lambda_4=-1 \ . \ed
Therefore $\psi_3(x)=x+1$ is the irreducible factor of $a_5(x)$ corresponding to  $\Phi_3(x)$.
We end up with the integer factorization of $a_5(x)$ into irreducible factors:
\bd a_5(x)=x(x+1)(x-1)(x^2-3)  \ . \ed
Since $U_5(x)=a_5(2x)$ we obtain the factorization of $U_5(x)$:
\bd U_5(x)=2x (2x+1)(2x-1) (4x^2-3)  \ . \ed

\end{example}

In the following list we give the polynomial $\psi_n$ corresponding to each cyclotomic polynomial $\Phi_n$  up to $n=24$.
\bd
 \begin{array}{lcl}

\psi_3(x) & = &x+1 \\
\psi_ 4(x) & =& x  \\
\psi_ 5(x) & = &x^2+x-1 \\
\psi_ 6(x) & =& x-1\\
\psi_ 7(x) & =& x^3+x^2-2x-1  \\
\psi_8 (x) & =&x^2-2  \\
\psi_9 (x) & =&  x^3-3x+1 \\
\psi_{10} (x) & =& x^2-x-1\\
\psi_{11} (x) & =&x^5+x^4-4x^3-3x^2+3x+1  \\
\psi_{12} (x) & = &x^2-3 \\
\psi_{13} (x) & =& x^6+x^5-5x^4-4x^3+6x^2+3x-1 \\
\psi_{14} (x) & = &x^3-x^2-2x +1 \\
\psi_{15} (x) & =& x^4-x^3-4x^2+4x+1\\
\psi_{16} (x) & =& x^4-4x^2+2\\
\psi_{17} (x) & =& x^8+x^7-7x^6-6 x^5 +15 x^4+10x^3-10x^2-4x+1 \\
\psi_{18} (x) & = &x^3-3x-1 \\
\psi_{19} (x) & = &x^9+x^8-8 x^7-7x^6+21x^5+15x^4-20 x^3-10 x^2+5x+1  \\
\psi_{20} (x) & =& x^4-5 x^2+5 \\
\psi_{21} (x) &=& x^6-x^5-6x^4+6 x^3+8x^2-8x +1  \\
\psi_{22} (x) &=& x^5-x^4-4 x^3+3x^2 +3x -1  \\
\psi_{23}  (x) &=& x^{11}+x^{10}-10x^9-9x^8+36 x^7+28 x^6-56 x^5 -35 x^4+35 x^3+15 x^2-6x -1 \\
\psi_{24} (x) &=& x^4 -4 x^2 +1  \ .
\end{array}
 \ed

\begin{remark}
Since $\psi_n(x)$ is the minimal polynomial of $2 \cos \frac{2 \pi}{n}$ it is reasonable to define $\psi_1(x)=x-2$ and $\psi_2(x)=x+2$. They correspond to the reducible $Q_1(x)=(x-1)^2$ and $Q_2(x)=(x+1)^2$ respectively.
\end{remark}

 \begin{remark} \label{ex3660}
To compute $\psi_n$ is straightforward.  We give two examples.

\begin{itemize}
\item
$n=36$.
Since
\bd \Phi_{36}=x^{12}-x^6+1 \ \ed
the polynomial $\psi_{36} $ is of degree $6$.
We need
\bd x^{12}-x^6+1=x^6 \psi_{36}(\zeta)  \ed
where $\zeta=x+ \frac{1}{x}$.
Therefore
\bd \psi_{36}( \zeta)=x^6+ \frac{1}{x^6}-1=(\zeta^6-6 \zeta^4+9  \zeta^2-2)-1 = \zeta^6-6 \zeta^4+9  \zeta^2-3 \ .  \ed
\item

$n=60$.
Since
\bd \Phi_{60}=x^{16}+x^{14}-x^{10}-x^{8}-x^6 +x^2+1 \ \ed
the polynomial $\psi_{60} $ is of degree $8$.
We need
\bd   x^{16}+x^{14}-x^{10}-x^{8}-x^6 +x^2+1  =x^8 \psi_{60}(\zeta)  \ed
where $\zeta=x+ \frac{1}{x}$.
Therefore
\begin{align*}
 \psi_{60}( \zeta)&=(x^8+\frac{1}{x^8})+ (x^6+ \frac{1}{x^6}) -(x^2+\frac{1}{x^2})-1 \\
  &=  a_8 (\zeta)+ a_6(\zeta)-a_2(\zeta)-1 \\
  &=\zeta^8-7 \zeta^6 +14 \zeta^4-8 \zeta^2+1  \ .  \end{align*}

\end{itemize}

\end{remark}

\begin{example}
To find the factorization of \bd U_9(x)=512 x^9 -1024 x^7+672 x^5 -160 x^3 +10x \ . \ed

Since
\bd Q_9(x)=\Phi_4 \Phi_5 \Phi_{10} \Phi_{20}  \ , \ed
we have that
\bd a_9(x)=\psi_4 \psi_5 \psi_{10} \psi_{20} =x (x^2+x-1)(x^2-x-1)(x^4-5 x^2+5) \ . \ed
Finally
\bd U_9(x)=a_9(2x)=2x(4x^2+2x-1)(4x^2-2x-1)(16x^4-20x^2+5)  \ . \ed

\end{example}

To conclude we state the following result:

\begin{proposition}

\bd  U_n(x)=\prod_{\substack{ j|2n+2 \\  j\not=1,2}} \psi_j(2x) \ . \ed

\end{proposition}

\subsection{Chebyshev polynomials of the first kind}
The roots of $T_n$ are given by
\bd x_k=\cos {\frac{(2k-1) \pi}{2n} }  \ \ \ \ \  \ k=1,2, \dots, n  \ . \ed
Let $h=2n$  ($h$ is the Coxeter number of the root system $B_n$) and $\theta=\frac{\pi}{h}$.
Then
\bd x_k=\cos (2k-1) \theta  \ \ \ \ \  \ k=1,2, \dots, n  \ ,  \ed
i.e. the roots of $T_n$ are
\bd \cos k \theta \ed
where $k$ runs over the exponents of a root system of type $B_n$.

\begin{itemize}
\item $a_n(x)$

The roots of $a_n(x)=2 T_n( \frac{x}{2})$ are
\bd \lambda_k=2 \cos {  \frac{(2k-1) \pi}{h}     }=2 \cos {(2k-1) \theta}   \ \ \ \ \  \ k=1,2, \dots, n  \ . \ed

Denote them by
\bd \lambda_1=2\cos \theta_1=2\cos \theta, \ \ \  \lambda_2=2\cos \theta_2, \dots, \ \ \lambda_n=2\cos \theta_n   \ .\ed
Note that $\theta_k=(2k-1)\theta$ and
\bd \theta_j + \theta_{n+1-j}=\pi  \ . \ed
This implies that $\lambda_j=-\lambda_{n-j+1}$.  As a result
\be \label{minusb}  \{ \lambda_j | j=1,2, \dots, n \}=\{-\lambda_j| j=1,2, \dots, n \}  \ . \ee

\item
$p_n(x)$

The roots of $p_n(x)$ are then
\bd \xi_k=2 + \lambda_k=2+ 2 \cos { \frac{(2k-1)\pi }{2n}}=2 +2\cos \theta_k=4 \cos^2 \frac{\theta_k}{2}  \ . \ed
It follows from (\ref{minusb}) that the eigenvalues of $C$ occur in pairs $\{ 2+ \lambda, 2-\lambda \}$.

\item
$Q_n(x)$

It follows from Theorem \ref{kron} that the roots of $Q_n(x)$ are
\bd e^{i\theta_1}, e^{i\theta_2}, \dots, e^{i\theta_n}, e^{-i \theta_1}, e^{-i\theta_2}, \dots, e^{-i\theta_n} \ , \ed
or equivalently,
\bd e^{i\theta}, e^{3i\theta}, e^{5i \theta}, \dots, e^{(2n-1)i\theta}, e^{(2n+1)i \theta}, e^{(2n+3)i \theta}, \dots, e^{(4n-1)i \theta} \ . \ed
Note that \bd e^{i \theta}=e^{ \frac{i \pi}{h}} =e^{  \frac{2 \pi i}{2h} } =e^{  \frac{2 \pi i}{4n} } \ . \ed
As a result  $e^{i \theta}$ is a $(2h)$th primitive root of unity and therefore  a root of the cyclotomic polynomial $\Phi_{2 h}=\Phi_{4n}$.  The other roots of this cyclotomic polynomial are of course $e^{ki \theta}$  where $(k, 2h)=1$ and $k$ odd. This determines $\Phi_{2h}$ as an irreducible factor of $Q_n(x)$. To determine the other irreducible factors we proceed as follows: Take an odd  divisor $d$ of $4n$. It is of course an odd divisor of $n$ as well.  If we write $n=2^{\alpha} N$ where $N$ is odd, this divisor $d$ is also a divisor of $N$. Use the notation $r=2^{\alpha+2}$ and note that  $2h=4n=rN$.  We group together all $e^{ik \theta}$ where $k$ is odd and $(N, k)=\frac{N}{d}$.  Note that $d=N$ corresponds to $\Phi_{4n}$.  These roots define $\Phi_{rd}$.
You can consider this argument as another proof of Proposition  \ref{prodbn}.

\item
$a_n(x)$

 As in the case of $A_n$  the roots of $a_n(x)$ are of the form $e^{i\theta_k}+ e^{-i \theta_k}$, where $e^{i \theta_k}$ is a root of $Q_n(x)$. Again, the irreducible factor of $Q_n(x)$  are in  one-to-one correspondence with the irreducible factors of $a_n(x)$.  We denote the irreducible factors of $a_n(x)$ with $\psi_n$ as before.

\end{itemize}
\begin{example}
To determine the factorization of  \bd T_5(x)=T_ 5(x) = 16x^5-20x^3+5x  \ . \ed    Since $n=5$, $h=2n=10$ and $\theta=\frac{\pi}{10}$.
The roots of $T_5$ are

\bd  \cos \frac{\pi}{10} \ \ \  \cos\frac{3 \pi}{10} \ \ \  \cos \frac{5\pi}{10} \ \ \  \cos \frac{ 7 \pi}{10} \ \ \  \cos \frac{9 \pi}{10}  \ .  \ed

The  roots of $a_5(x)$ are

\bd  \lambda_1=2 \cos \frac{\pi}{10} \ \ \  \lambda_2=2\cos\frac{3 \pi}{10} \ \ \ \lambda_3= 2\cos \frac{5\pi}{10} \ \ \  \lambda_4=2\cos \frac{ 7 \pi}{10} \ \ \  \lambda_5=2\cos \frac{9 \pi}{10}  \ .  \ed

Therefore, the  roots of $Q_5(x)$ are
\bd e^{ \frac{\pi }{10} }, e^{  \frac{3 \pi }{10} }, e^{  \frac{5 \pi }{10} }, e^{  \frac{7 \pi }{10} }, e^{  \frac{ 9\pi }{10} }, e^{ \frac{11\pi }{10} }, e^{ \frac{13 \pi }{10} }, e^{ \frac{ 15 \pi }{10} }, e^{ \frac{17 \pi }{10} }, e^{19 \frac{\pi }{10} } \ . \ed

We group together
\bd e^{ \frac{\pi }{10} }, e^{ \frac{3\pi }{10} }, e^{ \frac{7 \pi }{10} }, e^{ \frac{9 \pi }{10} }, e^{ \frac{11 \pi }{10} }, e^{ \frac{13 \pi }{10} }, e^{ \frac{17 \pi }{10} }, e^{ \frac{19 \pi }{10} }\  \ed
which are the roots of $\Phi_{20}(x)$. These are all the exponentials $e^{ik\theta}$ with $k$ odd and $(k, 20)=1$.

 Note that
\bd e^{ \frac{\pi }{10} }+e^{19 \frac{\pi }{10} }=e^{ \frac{\pi }{10} }+e^{- \frac{\pi }{10} }=2 \cos \frac{\pi}{10}=\lambda_1  \ , \ed

\bd e^{ \frac{3 \pi }{10} }+e^{ \frac{17\pi }{10} }=e^{ \frac{3 \pi }{10} }+e^{ \frac{-3 \pi }{10} }=\cos \frac{3 \pi}{10}=\lambda_2 \ ,  \ed

\bd e^{ \frac{7 \pi }{10} }+e^{ \frac{13\pi }{10} }=e^{ \frac{7 \pi }{10} }+e^{ \frac{-7 \pi }{10} }=\cos \frac{7 \pi}{10}=\lambda_4 \ ,  \ed

\bd e^{ \frac{9 \pi }{10} }+e^{ \frac{11\pi }{10} }=e^{ \frac{9 \pi }{10} }+e^{ \frac{-9 \pi }{10} }=\cos \frac{9 \pi}{10}=\lambda_5 \ ,  \ed

These roots $\lambda_1$, $\lambda_2$, $\lambda_4$ and $\lambda_5$  are roots of $\psi_{20}=x^4-5 x^2+5$ which is an irreducible factor of $a_5(x)$.

The only other odd divisor of $20$ is $5$. Therefore we group together
\bd e^{ \frac{5 \pi }{10} }, e^{ \frac{15 \pi }{10} } \ed
which are the roots of $\Phi_4(x)$. These are all the exponentials $e^{ik\theta}$ with $k$ odd and $(k, 20)=5$. Noting that $20=2^2 \cdot 5$ we see that $1,5$ are just the positive divisors of $5$.

 Note that
\bd e^{ \frac{5\pi }{10} }+ e^{ \frac{15 \pi }{10} }=e^{ \frac{5 \pi }{10} }+ e^{- \frac{5 \pi }{10} }=2\cos\frac{5 \pi}{10}=\lambda_3=0\ . \ed
Therefore $x$ is the irreducible factor of $a_5(x)$ corresponding to  $\Phi_4(x)$.

We end up with the integer factorization of $a_5(x)$ into irreducible factors:
\bd a_5(x)=x(x^4-5 x^2+5)  \ . \ed
Since $T_5(x)=\frac{1}{2}a_5(2x)$ we obtain the factorization of $T_5(x)$:
\bd T_5(x)=x(16x^4-20 x^2+5)  \ . \ed

\end{example}

\begin{example}
To find the factorization of \bd T_9(x)=256 x^9 -576 x^7+432 x^5 +20 x^3+9x \ . \ed

Since
\bd Q_9(x)=\Phi_4 \Phi_{12} \Phi_{36}  \ , \ed
we have that
\bd a_9(x)=\psi_4 \psi_{12} \psi_{36}  \ . \ed
Since $\psi_4(x)=x$ and $\psi_{12}(x)=x^2-3$ and  $\psi_{36}=x^6-6 x^4+9x^2-2$ (Remark \ref{ex3660})
we have
\bd a_9(x)=x(x^2-3)(x^6-6x^4+9 x^2-3)  \ . \ed
Finally
\bd T_9(x)=\frac{1}{2} a_9(2x)= x(4x^2-3)(64x^6-96x^4+36x^2-3)  \ . \ed

\end{example}

To conclude we state the following result:

\begin{proposition}

Let $n=2^{\alpha} N$ where $N$ is odd and let $r=2^{\alpha+2}$. Then

\bd  T_n(x)=\frac{1}{2}\prod_{\substack{ j|N \\   }} \psi_{r j}(2x) \ . \ed

\end{proposition}

\begin{remark}
In the case of $D_n$
\bd q_n(x)=4x T_{n-1}(x)   \ .\ed
Therefore
\bd a_n(x)=2x T_{n-1}(x) \ .\ed

\bd a_n(x_0)=0 \ \ \ \ \ \ \  \Rightarrow \ \ \ \ \ \ \ \ \ \  \ \ 2 x_0 T_{n-1}(x_0)=0    \ed
\bd \Rightarrow \ \ \  x_0=0, \ \ \ \ {\rm or} \ \ \ \ \ \ \ \ \  x_0=2 \cos \frac {(2k-1) \pi}{2(n-1)}  \ \ \ k=1,2, \dots, n-1  \ .\ed
In summary: The roots of $a_n(x)$ are
\bd 2 \cos \frac{m_i \ \pi} {h}  \ed
where $h$ is the Coxeter number for $D_n$ and $m_i$ are the exponents.  It follows that $0$ is always a root and $a_n(x)=x g_n(x)$ where $g_n(x)$ is the $a_{n-1}$ characteristic polynomial for $B_{n-1}$.
\end{remark}

\section{Exceptional Lie algebras} \label{cex}

\subsection{ $G_2$  }

The Cartan matrix for  $G_2$ is
\bd \begin{pmatrix} 2 & -1 \\
-3& 2 \end{pmatrix} \ . \ed

The characteristic polynomial for a Lie algebra of type $G_2$ is
\bd p_2(x)=x^2-4x +1  \ , \ed
since
\bd q_2(x)=2 T_2(x)- U_0(x)=4x^2 -3 \ed
and
\bd p_2(x)=q_2( \frac{x}{2}-1) =x^2-4x+1  \ . \ed
The roots of $a_2(x)=x^2-3$ are
\bd 2 \cos \frac{ m_i \pi}{h}  \ed
where $m_1=1$ and $m_2=5$ are the exponents of root system of type $G_2$. The Coxeter number $h$ is $6$.

Finally,
\bd Q_2(x)=x^4-x^2+1=\Phi_{12}(x) \ ,  \ed

and \bd f_2(x)=x^2-x+1 =\Phi_6(x)  \ . \ed

\subsection{Graph of type  $F_4$}

The Cartan matrix for $F_4$  is
\bd \begin{pmatrix}  2 & -1 & 0 & 0 \\
-1&2&-2&0 \\
0&-1&2&-1 \\
0&0&-1& 2 \end{pmatrix} \ . \ed

\bd p_4(x)=x^4-8 x^3+20x^2-16 x+1 \ , \ed

and
\bd a_4(x)=x^4-4x^2+1 =\psi_{24}(x) \ . \ed
The roots of $a_4(x)$ are
\bd \frac{1}{2} (\pm \sqrt{6} \pm \sqrt{2} )  \ed
i.e.
\bd 2 \cos  \frac{m_i \ \pi }{12}   \ed
where $m_i  \in \{1,5, 7, 11 \} $.   These are the exponents for $F_4$ and being the numbers less than $12$ and prime to $12$ imply

\bd f_4(x)=x^4-x^2+1=\Phi_{12}(x) \ . \ed

\subsection{ $E_n$ graphs}

\begin{itemize}
\item

 $n=6$

The Cartan matrix for $E_6$ is
\bd
\begin{pmatrix}
2&0&-1& 0& 0 & 0 \\
0&2&0& -1& 0 & 0 \\
-1&0&2& -1& 0 & 0 \\
0&-1&-1& 2& -1 & 0 \\
0&0&0& -1& 2 & -1 \\
0&0&0& 0& -1 & 2
\end{pmatrix}
\ . \ed

\bd q_6(x)=64x^6-80 x^4+20x^2-1 =(2x+1)(2x-1)(16x^4-16x^2+1)   \ed
\bd p_6(x)=(x-1)(x-3)(x^4-8x^3+20x^2-16x +1)  \ed
\bd a_6(x)=x^6-5 x^4+5 x^2-1=(x+1)(x-1)(x^4-4x^2+1)=\psi_3(x) \psi_6(x) \psi_{24}(x) \ed
\bd Q_6(x)=(x^2+x+1)(x^2-x+1)(x^8-x^4+1)  =\Phi_3(x) \Phi_6(x) \Phi_{24}(x) \ .  \ed

The exponents of $E_6$ are $\{1,4,5,7,8,11\}$ and the Coxeter number is $12$. The subset $\{1,5,7,11\}$ produces $\Phi_{12}$ and $\{4,8\}$ produces $\Phi_3$. Therefore
\bd f_6(x)=\Phi_3(x) \Phi_{12}(x)  \ .\ed

The roots of $a_6(x)$ are

\bd \pm 1,  \ \  \frac{1}{2} (\pm \sqrt{6} \pm \sqrt{2} )  \ed
i.e.
\bd 2 \cos  \frac{ m_i \ \pi }{12}   \ed
where $m_i  \in \{1, 4, 5, 7, 8, 11 \} $.   These are the exponents for $E_6$. The Coxeter number is $12$.

\item

$n=7$

The Cartan matrix for $E_7$ is
\bd \begin{pmatrix}
2 &0  &-1  &0 &0 & 0& 0  \\
 0 & 2 & 0 &-1 &0 &0 &0   \\
   -1&0  & 2 &-1 &0 &0 &0   \\
    0&-1  &-1  &2 & -1&0 & 0  \\
    0 & 0 & 0 &-1 &2 &-1 & 0  \\
     0 & 0 &0  &0 &-1 &2 &-1   \\
      0 &0  &0  &0 & 0&-1 &  2
      \end{pmatrix} \ . \ed

\bd q_7(x)=128x^7-192 x^5+72 x^3-6x =2x (64x^6-96x^4+36 x^2-3)    \ed
\bd p_7(x)=(x-2)(x^6-12x^5+54 x^4-112x^3+105 x^2 +1)   \ed
\bd a_7(x)=x^7-6 x^5 +9x^3-3x =x(x^6-6 x^4+9 x^2-3)= \psi_4(x) \psi_{36}(x)  \ed
\bd Q_7(x)=(x^2+1)(x^{12}-x^6 +1)=\Phi_4(x)  \Phi_{36}(x)  \ . \ed

The exponents of $E_7$ are $\{1,5,7,9,11,13,17\}$ and the Coxeter number is $18$. The subset $\{1,5,7,11,13,17\}$ produces $\Phi_{18}$ and $\{ 9\}$ produces $\Phi_2$. Therefore
\bd f_7(x)=\Phi_2(x) \Phi_{18}(x)  \ .\ed

\item

$n=8$

The Cartan matrix for $E_8$ is
\bd \begin{pmatrix}
 2& 0 &-1  &0 &0 & 0&0 & 0 \\
 0 & 2 &0  &-1 &0 &0 & 0& 0 \\
  -1 &0  &2  &-1 &0 &0 &0 &0  \\
   0 &-1  &-1  & 2&-1 &0 &0 &0  \\
    0 &0  &0  &-1 &2 &-1 &0 &0  \\
     0 &0  &0  &0 &-1 &2 &-1 &0  \\
      0 & 0 & 0 &0 &0 &-1 &2 &-1  \\
        0& 0 &0  &0 &0 &0 &-1 &2  \\
        \end{pmatrix}
        \ . \ed

\bd q_8(x)=256 x^8 -448 x^6 +224 x^4 -32 x^2 +1    \ed
\bd p_8(x)=x^8-16x^7+105 x^6 +364 x^5 +714 x^4 -784 x^3 +440 x^2 -96 x+1 \ed
\bd a_8(x)=x^8-7x^6+14 x^4-8x^2+1=\psi_{60} (x)\ed
\bd Q_8(x)= x^{16} +x^{14} -x^{10}-x^8-x^6+x^2+1 = \Phi_{60}(x)  \ . \ed
The exponents of $E_8$ are $\{1,7,11,13,17,19,23,29 \}$ which are the positive integers less than $30$ and prime to $30$.  Therefore
\bd f_8(x)= \Phi_{30}(x)  \ .\ed
\end{itemize}

\section{The sine formula}

To prove the sine formula we only  will need the following Lemma which was  already proved by case to case verification:

 \begin{lemma} \label{rootscoxeter}  The  roots of $a_n(x)$ are
\bd 2 \cos \frac{ m_i \pi}{h}   \ed
where $m_i$ are the exponents of $\mathfrak{g}$ and $h$ is the Coxeter number of $\mathfrak{g}$.
\end{lemma}

 Let $\theta_i=\frac{ m_i \pi}{h} $. Then the roots of $a_n(x)$ are $\lambda_i=2 \cos \theta_i, \ \ i=1,2, \dots, \ell$.

Recall the duality property of the exponents (\ref{dual}).
\bd m_i+m_{\ell+1-i} = h  \  . \ed

It follows that
\bd m_i \frac{\pi}{h} + m_{\ell+1-i} \frac{\pi}{h}= \pi  \ .\ed
As a result:
\bd \theta_i+\theta_{\ell+1-i}=\pi  \ . \ed

Using this formula,  we can infer  a relationship satisfied by the  roots of $p_n(x)$ which are:
\bd \xi_i=4 \cos^2 \frac{\theta_i}{2}  \ . \ed

Namely:
\bd
 \begin{array}{lcl}
 \xi_i+\xi_{\ell+1-i}&=&4 \cos^2 \frac{\theta_i}{2}+4 \cos^2 \frac{\theta_{\ell+1-i}}{2}\\
                       &=& 4 \left( \cos^2 \frac{\theta_i}{2}+\cos^2 \frac{\pi- \theta_i}{2} \right)\\
                       &=&4 \left( \cos^2 \frac{\theta_i}{2}+\sin^2 \frac{ \theta_i}{2} \right)  \\
                       &=& 4     \end{array}  \ed

 It follows that the eigenvalues of the Cartan matrix occur in pairs $\xi$, $4 -\xi$.

\begin{remark} The fact that the  eigenvalues of $C$ occur in pairs  $\{  \xi, 4-\xi \}$, and a different line of proof  can be found in  \cite[p. 345]{moody}.
\end{remark}

\begin{theorem}
Let  $\mathfrak{g}$ be a complex simple Lie algebra of rank $\ell$,  $h$ the Coxeter number,  $m_1, m_2, \dots, m_{\ell}$ the exponents of $\mathfrak{g}$  and $C$ the Cartan matrix.  Then

\bd 2^{2 \ell} \prod_{i=1}^{\ell}  \sin^2 {\frac{ m_i \pi}{2h}     } ={\rm det}\ C  \ . \ed
\end{theorem}

\begin{proof}
It follows from Lemma (~\ref{rootscoxeter}) that

\bd a_n(x)=  \prod_{i=1}^{\ell} \left( x- ( 2 \cos  \frac{ m_i \ \pi }{h} )\right)  \ . \ed
Set $x=2$.

\begin{align*}
a_n(2)&= \prod_{i=1}^{\ell} \left( 2- ( 2 \cos  \frac{ m_i \ \pi }{h}) \right) \\
&=2^l \prod_{i=1}^{\ell} \left(1 -\cos  \frac{ m_i \ \pi }{h} \right)  \\
&=2^l \prod_{i=1}^{\ell} 2 \sin^2  \frac{ m_i \ \pi }{2h} \\
&=2^{2l} \prod_{i=1}^{\ell}  \sin^2  \frac{ m_i \ \pi }{2h}  \ .
\end{align*}

To prove the formula we calculate $a_n(2)$. We have
\bd p_n(x)=(x-\xi_1)(x-\xi_2) \dots (x-\xi_{\ell})= (x-(4-\xi_1))(x-(4-\xi_2)) \dots (x-(4-\xi_{\ell})) \ . \ed
This implies that $p_n(4)=\xi_1 \xi_2 \dots \xi_{\ell}=  {\rm det} \ C$.
Since $a_n(x)=p_n(x+2)$ we have
\bd a_n(2)=p_n(4)= {\rm det} \ C \ .\ed
\end{proof}

\bibliographystyle{amsplain}

\begin{thebibliography}{99}





\bibitem{aCampo}   N. A'Campo, Sur les valeurs propres de la transformation de Coxeter, \textsl{ Invent. Math.}  \textbf{ 33},  (1976), 61–-67.

\bibitem{moody} S. Berman, Y. S. Lee, R. V. Moody, The spectrum of a Coxeter transformation, affine Coxeter transformations and the defect map, \textsl{ J. Algebra}  \textbf{ 121} , no. 2, (1989), 339-–357.

\bibitem{coleman}  A. J. Coleman, Killing and the Coxeter transformation of Kac-Moody algebras, \textsl{ Invent. Math.} \textbf{ 95}, no. 3, (1989), 447-–477.

\bibitem{coleman2}  A. J.  Coleman,  The Betti numbers of the simple Lie groups. \textsl{Canad. J. Math.} 10 (1958) 349--356


\bibitem{collin} D. H. Collingwood, W. M.  McGovern,  Nilpotent orbits in semisimple Lie algebras. Van Nostrand Reinhold Mathematics Series. Van Nostrand Reinhold Co., New York, 1993.


\bibitem{coxeter}  H. S. M. Coxeter, Discrete groups generated by reflections, \textsl{ Ann. of Math. (2)} \textbf{ 35} (1934), 588–-621.

\bibitem{coxeter2} H. S. M. Coxeter,  The product of the generators of a finite group generated by reflections, \textsl{Duke Math. J.} \textbf{18} (1951) 765--782.

\bibitem{doobs} D. M.  Cvetkovic,  M.  Doob, H.  Sachs,  Spectra of graphs. Theory and application. Pure and Applied Mathematics, 87. Academic Press, New York, 1980.


\bibitem{damianou}    P. A. Damianou, Monic polynomials  in $Z[x]$ with roots in the unit
 disc, \textsl{ Amer. Math. Monthly}  \textbf{ 108} (2001),  253--257.


\bibitem{damianou2}  P. A. Damianou, On prime values of cyclotomic polynomials, \textsl{  Int. Math. Forum} \textbf{ 6}, no 29, (2011), 1445--1456.

\bibitem{damianou3} P. A. Damianou,  A beautiful sine formula, \textsl{ Amer. Math. Monthly}, \textbf{121}, (2014), 120--135.

\bibitem{jones}  F. M. Goodman, P. de la Harpe, and  V. F. R. Jones, \textsl{ Coxeter Graphs and Towers of Algebras}, Mathematical Sciences Research Institute Publications, 14, Springer-Verlag, New York, 1989.

\bibitem{hironaka}  E. Hironaka,   Lehmer's problem, McKay's correspondence, and $2,3,7$. Topics in algebraic and noncommutative geometry (Luminy/Annapolis, MD, 2001),  \textsl{Contemp. Math.}, \textbf{324}, Amer. Math. Soc., Providence, RI, (2003), 123-–138.

\bibitem{hsiao}  H. J. Hsiao, On factorization of Chebyshev's polynomials of the first kind, \textsl{
Bulletin of the Institute of Mathematics, Academia Sinica} \textbf{ 12} (1), (1984),  89--94.

\bibitem{humphreys} J. E. Humphreys,
\textsl{Introduction to Lie algebras and Representation Theory}, Springer,  New York, 1972.


\bibitem{howlett}  R. B. Howlett, Coxeter groups and M-matrices, \textsl{Bulletin of the London Mathematical Society}  \textbf{14} (1982), 137-–141.





\bibitem{kac} V.  G. Kac,
\textsl{Infinite dimensional Lie algebras}, Cambridge University Press, 1994.

\bibitem{knapp} A. W.  Knapp,   \textit{Lie Groups Beyond an Introduction}. Progress in Mathematics, 140. Birkh\"{a}user Boston,  Boston, MA, 1996.



\bibitem{koranyi} A. Koranyi,  Spectral properties of the Cartan matrices. Acta Sci. Math. (Szeged) 57 (1993), no. 1-4, 587–-592.

\bibitem{kostant} B. Kostant, The McKay correspondence, the Coxeter element and representation theory,  The mathematical heritage of Elie Cartan (Lyon, 1984), Asterisque 1985, Numero Hors Serie, 209–-255.

\bibitem{kostant2} B. Kostant,  The principal three-dimensional subgroup and the Betti numbers of a complex simple Lie group \textsl{Amer. J. Math.} \textbf{81}  (1959), 973--1032.


\bibitem{lakatos} P. Lakatos, Salem numbers defined by Coxeter transformation, Linear Algebra and Its Applications, 432 (2010), 144--154.

\bibitem{lehmer} D. H. Lehmer A Note on Trigonometric Algebraic Numbers. \textsl{ Amer. Math. Monthly} \textbf{40} (1933) 165--166.

\bibitem{Lenzing} H. Lenzing and J. de la Pena.   A Chebysheff recursion formula for Coxeter polynomials.
\newblock {\em Linear Algebra Appl.} 430,  (2009) 947–956.



\bibitem{moss}  M. Mossinghoff, Polynomials with small mahler measure, \textsl{ Mathematics of Computation} \textbf{ 67 } no. 224, (1998), 1697–-1705.

\bibitem{mason} J.C. Mason and D.C. Handscomb, \textsl{ Chebyshev Polynomials},  Chapman and Hall, 2002.

\bibitem{rayes} M. O.  Rayes, V. Trevisan, P. S.  Wang,  Factorization properties of Chebyshev polynomials, \textsl{ Comput. Math. Appl.} \textbf{ 50} (2005),  1231-–1240.


\bibitem{ringel}   C. M. Ringel, The spectral radius of the Coxeter transformations for a generalized Cartan
matrix,  \textsl{Math. Ann.} \textbf{300},  (1994),  331--339.

\bibitem{smith} G. D.  Smith,  Numerical solution of partial differential equations.
Finite difference methods. Third edition. The Clarendon Press, Oxford University Press, New York, 1985.



\bibitem{steinberg}  R. Steinberg, Finite subgroups of SU2, Dynkin diagrams and affine Coxeter elements, \textsl{ Pacific J.Math.} \textbf{118}  no. 2, (1985),  587-–598.


\bibitem{steko}  R. Stekolshchik,  \textsl{ Notes on Coxeter transformations and the McKay correspondence},  Springer Monographs in Mathematics. Springer-Verlag, Berlin, 2008.


\bibitem{wiki} Wikipedia,
\textsl{Chebyshev Polynomials}, Wikipedia, the free encyclopedia.
\end{thebibliography}

\end{document}